\documentclass[11pt]{amsart}
\usepackage{amsfonts}
\usepackage{amssymb}
\usepackage{amsmath}
\numberwithin{equation}{section}
\usepackage{dsfont}
\usepackage{color}
\usepackage{graphicx}
\newtheorem{theorem}{Theorem}[section]
\newtheorem{lemma}[theorem]{Lemma}
\newtheorem{proposition}[theorem]{Proposition}
\newtheorem{corollary}[theorem]{Corollary}

\theoremstyle{definition}
\newtheorem{df}{Definition}
\newtheorem{example}[df]{Example}
\newtheorem{remark}[df]{Remark}

\newcommand{\N}{\mathbb N}


\begin{document}
\author{Piotr Nowakowski}
\address{Institute of Mathematics, \L \'{o}d\'{z} University of Technology,
ul. W\'{o}lcza\'{n}ska 215, 93-005 \L \'{o}d\'{z}, Poland}
\email{piotr.nowakowski@dokt.p.lodz.pl}
\author{Tomasz Filipczak}
\address{Institute of Mathematics, Lodz University of Technology, ul. W\'{o}%
lcza\'{n}ska 215, 93-005 \L \'{o}d\'{z}, Poland}
\email{tomasz.filipczak@p.lodz.pl}
\title[Conditions for the difference set of a central Cantor set]{Conditions
for the difference set of a central Cantor set to be a Cantorval}
\date{}

\begin{abstract}
Let $C(\lambda )\subset \lbrack 0,1]$ denote the central Cantor set
generated by a sequence $\lambda =(\lambda _{n})\in \left( 0,\frac{1}{2}%
\right) ^{\mathbb{N}}$. By the known trichotomy, the difference set $%
C(\lambda )-C(\lambda )$ of $C(\lambda )$ is one of three possible sets: a
finite union of closed intervals, a Cantor set, and a Cantorval. Our main
result describes effective conditions for $(\lambda _{n})$ which guarantee
that $C(\lambda )-C(\lambda )$ is a Cantorval. We show that these conditions
can be expressed in several equivalent forms. Under additional assumptions,
the measure of the Cantorval $C(\lambda )-C(\lambda )$ is established. We
give an application of the proved theorems for the achievement sets of some
fast convergent series.
\end{abstract}

\maketitle

\section{Introduction}

By a \emph{Cantor set} we mean a compact, perfect and nowhere dense subset
of the real line.\ Cantor sets appear in several publications in many
different settings, cf. \cite{PP} and \cite{Ta}. They occur in mathematical
models involving fractals, iterated functional systems and fractional
measures \cite{Fl}. They play a role in number theory (e.g. in $b$-ary
number representations, and in connection with continued fractions \cite{H}%
). They are rooted in dynamical systems in the study of homoclinic
bifurcations \cite{PT}, in signal processes and ergodic theory, and also in
limit theorems from probability, as it is stated in \cite{PP}. Finally, let
us mention the applications in spectral theory \cite{DGS}, \cite{T16}.
Important results describe geometrical properties of Cantor sets via
Hausdorff and packing measures and the respective fractal dimensions, and
multifractal spectrum \cite{AS}, \cite{CHM}, \cite{GZ}, \cite{HZ}. Also,
several authors conducted extensive studies on the arithmetic sums \cite{AC}%
, \cite{MO}, \cite{T19} and products \cite{Ta} of two Cantor sets, and their
intersections \cite{HKY}, \cite{K1}.

In our paper we consider central Cantor sets $C\left( \lambda \right) $,
which we define using a sequence $\lambda =\left( \lambda _{n}\right) \in
\left( 0,\frac{1}{2}\right) ^{\mathbb{N}}$. These are Cantor sets such that
the ratio of measures of the intervals defined in subsequent steps of the
construction is equal to $\lambda _{n}$ (for the Cantor ternary set, $\lambda
_{n}\equiv \frac{1}{3}$). We are interested in properties of the difference
set $C(\lambda )-C(\lambda )$. Several authors examined sums and
differences of Cantor sets (see \cite{AI}, \cite{K1} \cite{K}, \cite{FF}, 
\cite{Ta}, \cite{S}). Considerations about sums and differences of central
Cantor sets were often based on the relationship between central Cantor sets
and the achievement sets (i.e. sets of all subsums) of convergent series
(see Proposition \ref{Pr2}).

The study of topological properties of the sets of subsums of series has
a long history.\ The first papers on this topic were written over a hundred
years ago. In 1914 Kakeya \cite{Ka} proved that, if $x_{n}>\sum_{j=n+1}^{%
\infty }x_{j}$ for any $n$, then the set $E\left( x\right) $ of all subsums
of a series $\sum_{j=1}^{\infty }x_{j}$ of positive numbers is a Cantor
set. He also showed that $E\left( x\right) $ is a finite union of closed
intervals if and only if $x_{n}\leq \sum_{j=n+1}^{\infty }x_{j}$ for almost
all $n$'s. Kakeya conjectured that the set $E\left( x\right) $ is either a
Cantor set or a finite union of closed intervals. The first examples showing
that this hypothesis is false appeared in the papers by Weinstein and
Shapiro \cite{WS}, and by Ferens \cite{Fe}. Guthrie and Nymann \cite{GN}
formulated the theorem that $E\left( x\right) $ is either a Cantor set or a
finite union of closed intervals or a Cantorval (compare also \cite{NS}). 
Nice descriptions of these problems can be found in \cite{N} and \cite{BFP}. 
Recently, there have appeared a lot of interesting papers that explore properties of
achievement sets, i.a. \cite{BBFS}, \cite{BP}, \cite{FF}, \cite{BGM}, and \cite%
{B}.

Based on results of sets of subsums of series, Anisca and Ilie in
\cite{AI} proved that a finite sum of central Cantor sets is either a
Cantor set or a finite union of closed intervals, or a Cantorval. In
particular, the set $C(\lambda )-C(\lambda )$ has this property (compare
Theorem \ref{AI}). Moreover, $C(\lambda )-C(\lambda )$ is a finite union of
closed intervals if and only if $\lambda _{n}\geq \frac{1}{3}$ for almost
all $n$'s (see Theorem \ref{tw1}). A main goal of our paper is to find
conditions which imply that the difference set $C\left( \lambda \right)
-C\left( \lambda \right) $ is a Cantorval. We examine sets $C\left( \lambda
\right) -C\left( \lambda \right) $ for the sequences such that $\lambda _{n}<%
\frac{1}{3}$\ for infinitely many terms and $\lambda _{n}\geq \frac{1}{3}$\
for infinitely many terms. The main result is Theorem \ref{Th1} containing
a sufficient condition for the set $C\left( \lambda \right) -C\left(
\lambda \right) $ to be a Cantorval. This theorem gives also a formula for
the measure of such a Cantorval (written as the sum of the series) and some
information about its interior. The proof of the theorem is based on the
properties resulting from the construction of central Cantor sets and it
does not use methods related to the study of achievement sets. However, we
use relationships between central Cantor sets and achievement sets in
Theorem \ref{th5} describing the sets that meet the assumptions of Theorem %
\ref{Th1}. In Theorem \ref{tw6} we get a simple formula for the measure of
considered Cantorvals written as achievement sets.

Let us introduce basic notation. For $A,B\subset \mathbb{R}$ we denote by $A\pm B$ the set $\left\{ a\pm
b:a\in A,\,b\in B\right\} $. The set $A-A$ is called the\emph{\ difference
set} of $A$. The Lebesgue measure of a measurable set $A$ is denoted by $%
\left\vert A\right\vert $. By $l\left( I\right) $, $r\left( I\right) $ and $%
c\left( I\right) $ we denote the left endpoint, the right endpoint and the
center of a bounded interval $I$, respectively. The sequence $(i,\dots ,i)$
with $n$ terms is denoted by $i^{(n)}$. By $t|n$ we denote the sequence
consisting of the first $n$ terms of a given sequence $t$. If $t|n=s$,
we say that $t$ is an \emph{extension} of $s$ and write $s\prec t$. To
denote the concatenation of two sequences $t$ and $s$, we write $t\symbol{94}%
s$.

The construction of a central Cantor subset of $[0,1]$ is the following (see 
\cite{BFN}).

Let $\lambda =\left( \lambda _{n}\right) $ be a sequence such that $\lambda
_{n}\in \left( 0,\frac{1}{2}\right) $ for any $n\in \mathbb{N}$. Let $I:=\left[
0,1\right] $ and $P$ be the open interval such that $c\left( P\right)
=c\left( I\right) $ and $\left\vert P\right\vert =\left( 1-2\lambda
_{1}\right) \left\vert I\right\vert $. The left and the right components of $%
I\setminus P$ are denoted by $I_{0}$ and $I_{1}$, we write $d_{1}:=\left\vert
I_{0}\right\vert =\left\vert I_{1}\right\vert $, and we define open intervals, of length $\left( 1-2\lambda _{2}\right) d_{1}$%
, that are concentric with $I_{0}$ and $I_{1}$ as $P_{0}$ and $%
P_{1}$, respectively. In general, $%
I_{t_{1},\dots ,t_{n},0}$ and $I_{t_{1},\dots ,t_{n},1}$ are the left and
the right components of the set $I_{t_{1},\dots ,t_{n}}\setminus
P_{t_{1},\dots ,t_{n}}$, $d_{n+1}$ is the length of each of these
components, and $P_{t_{1},\dots ,t_{n},0}$ and $P_{t_{1},\dots ,t_{n},1}$
are open intervals of length $\left( 1-2\lambda _{n+2}\right) d_{n+1}$,
concentric with $I_{t_{1},\dots ,t_{n},0}$ and $I_{t_{1},\dots ,t_{n},1}$,
respectively.

From the construction it follows that the length of each interval $%
I_{t_{1},\dots ,t_{n}}$ is equal to 
\begin{equation*}
d_{n}=\lambda _{1}\cdot \ldots \cdot \lambda _{n}.
\end{equation*}%
Denote $C_{n}(\lambda ):=\bigcup_{(t_{1},\dots ,t_{n})\in
\{0,1\}^{n}}I_{t_{1},\dots ,t_{n}}$ and $C(\lambda ):=\bigcap_{n\in \mathbb{N%
}}C_{n}(\lambda )$. Then $C(\lambda )$ is called a \emph{central Cantor set}.

Let $J:=I-I=\left[ -1,1\right] $. For fixed $n\in \mathbb{N}$ and $s\in
\left\{ 0,1,2\right\} ^{n}$ we define the interval $J_{s}$ by $%
J_{s}:=I_{t}-I_{p}$, where $p,t\in \left\{ 0,1\right\} ^{n}$ satisfy $%
s_{i}=t_{i}-p_{i}+1$ for $i=1,\ldots ,n$. Observe that the definition of $%
J_{s}$ does not depend on the choice of $p\ $and $t$. Indeed, for $n=1$ we
have $J_{0}=I_{0}-I_{1}=\left[ -1,-1+2d_{1}\right] $, $%
J_{1}=I_{0}-I_{0}=I_{1}-I_{1}=\left[ -d_{1},d_{1}\right] $ and $%
J_{2}=I_{1}-I_{0}=\left[ 1-2d_{1},1\right] $. For $n\in \mathbb{N}$, $s\in
\left\{ 0,1,2\right\} ^{n}$ and $p,t\in \left\{ 0,1\right\} ^{n}$ such that $%
J_{s}=I_{t}-I_{p}$, we have%
\begin{eqnarray*}
J_{s\symbol{94}0} &=&I_{t\symbol{94}0}-I_{p\symbol{94}1}=\left[ l\left(
I_{t}\right) ,l\left( I_{t}\right) +d_{n+1}\right] -\left[ r\left(
I_{p}\right) -d_{n+1},r\left( I_{p}\right) \right]  \\
&=&\left[ l(J_{s}),l(J_{s})+2d_{n+1}\right] , \\
J_{s\symbol{94}1} &=&I_{t\symbol{94}0}-I_{p\symbol{94}0}=I_{t\symbol{94}%
1}-I_{p\symbol{94}1}=[c(J_{s})-d_{n+1},c(J_{s})+d_{n+1}], \\
J_{s\symbol{94}2} &=&I_{t\symbol{94}1}-I_{p\symbol{94}%
0}=[r(J_{s})-2d_{n+1},r(J_{s})],
\end{eqnarray*}%
which proves the correctness of the definition of $J_{s}$.

Let $n\in \mathbb{N}$, $s\in \left\{ 0,1,2\right\} ^{n}$, and $\lambda \in
\left( 0,\frac{1}{2}\right) ^{\mathbb{N}}$. If $\lambda _{n+1}<\frac{1}{3}$,
then the set $J_{s}\setminus \left( J_{s\symbol{94}0}\cup J_{s\symbol{94}%
1}\cup J_{s\symbol{94}2}\right) $ is a union of two open intervals. We
denote them by $G_{s}^{0}$ and $G_{s}^{1}$, and call the \emph{left} and the 
\emph{right gap} in $J_{s}$. If $\lambda _{n+1}\geq \frac{1}{3}$, then $J_{s%
\symbol{94}0}\cap J_{s\symbol{94}1}\neq \emptyset $ and $J_{s\symbol{94}%
1}\cap J_{s\symbol{94}2}\neq \emptyset $. We denote these intervals by $%
Z_{s}^{0}$ and $Z_{s}^{1}$, and we call them the \emph{left} and the \emph{%
right overlap} in $J_{s}$. We also assume that $d_{0}:=\left\vert
I\right\vert =1$ and that $\left\{ 0,1,2\right\} ^{0}$ contains only the
empty sequence $s=\emptyset $. For $s\in \left\{ 0,1,2\right\} ^{n}$ ($n\in 
\mathbb{N}$) we define the numbers 
\begin{equation*}
N\left( s,0\right) :=\max \left\{ j:s_{j}>0\right\} \text{ and }N\left(
s,1\right) :=\max \left\{ j:s_{j}<2\right\} ,
\end{equation*}%
where $\max \emptyset :=0$. From the above definitions we conclude the
following easy properties.

\begin{proposition}
\label{lem}Let $n,k\in \mathbb{N}$, $s,u\in \left\{ 0,1,2\right\} ^{n}$, and $%
\lambda \in \left( 0,\frac{1}{2}\right) ^{\mathbb{N}}$. The following
properties hold.

\begin{enumerate}
\item \label{1-1}$\left\vert J_{s}\right\vert =2d_{n}.$

\item \label{1-2}$l\left( J_{s\symbol{94}0^{\left( k\right) }}\right)
=l\left( J_{s}\right) ,c\left( J_{s\symbol{94}1^{\left( k\right) }}\right)
=c\left( J_{s}\right) $, and $r\left( J_{s\symbol{94}2^{\left( k\right)
}}\right) =r\left( J_{s}\right) .$

\item \label{1-2b}If $n>k$, then $d_{n-1}-d_{n}<d_{k-1}-d_{k}$.

\item \label{1-2a}$l\left( J_{s}\right) -l\left( J_{u}\right) =c\left(
J_{s}\right) -c\left( J_{u}\right) =r\left( J_{s}\right) -r\left(
J_{u}\right) =\sum_{r=1}^{n}\left( s_{r}-u_{r}\right) \cdot \left(
d_{r-1}-d_{r}\right) .$

\item \label{1-3}If $\lambda _{n+1}<\frac{1}{3}$, then $G_{s}^{0}=\left[
r\left( J_{s\symbol{94}0}\right) ,l\left( J_{s\symbol{94}1}\right) \right] $%
, $G_{s}^{1}=\left[ r\left( J_{s\symbol{94}1}\right) ,l\left( J_{s\symbol{94}%
2}\right) \right] $, and $\left\vert G_{s}^{0}\right\vert =\left\vert
G_{s}^{1}\right\vert =d_{n}-3d_{n+1}.$

\item \label{1-4}If $\lambda _{n+1}\geq \frac{1}{3}$, then $Z_{s}^{0}=\left[
l\left( J_{s\symbol{94}1}\right) ,r\left( J_{s\symbol{94}0}\right) \right] $%
, $Z_{s}^{1}=\left[ l\left( J_{s\symbol{94}2}\right) ,r\left( J_{s\symbol{94}%
1}\right) \right] $, and $\left\vert Z_{s}^{0}\right\vert =\left\vert
Z_{s}^{1}\right\vert =3d_{n+1}-d_{n}.$

\item \label{1-5}$C_{n}\left( \lambda \right) -C_{n}\left( \lambda \right)
=\bigcup_{t\in \left\{ 0,1,2\right\} ^{n}}J_{t}.$

\item \label{1-6}If $\lambda _{n+1}\geq \frac{1}{3},\ldots ,\lambda
_{n+k}\geq \frac{1}{3}$, then $J_{s}=\bigcup_{t\in \left\{ 0,1,2\right\}
^{n+k},s\prec t}J_{t}$ and $C_{n}\left( \lambda \right) -C_{n}\left( \lambda
\right) =C_{n+k}\left( \lambda \right) -C_{n+k}\left( \lambda \right) $.

\item \label{1-7}$C(\lambda )-C(\lambda )=\bigcap_{n\in \mathbb{N}%
}(C_{n}(\lambda )-C_{n}(\lambda ))$.

\item \label{1-8}$C(\lambda )+C(\lambda )=\left( C(\lambda )-C(\lambda
)\right) +1$.
\end{enumerate}

\begin{proof}
Conditions (\ref{1-1})-(\ref{1-2}), (\ref{1-3})-(\ref{1-4})\ and the
equality $l\left( J_{s}\right) -l\left( J_{u}\right) =c\left( J_{s}\right)
-c\left( J_{u}\right) =r\left( J_{s}\right) -r\left( J_{u}\right) $ are
obvious. Since $\lambda _{n}<\frac{1}{2}$, we have $%
d_{n-1}-d_{n}>d_{n}>d_{n}-d_{n+1}$, which implies (\ref{1-2b}). If $t\in
\left\{ 0,1,2\right\} ^{r}$, where $r\in \mathbb{N}\cup \left\{ \emptyset
\right\} $, then%
\begin{equation*}
c\left( J_{t\symbol{94}1}\right) -c\left( J_{t\symbol{94}0}\right) =c\left(
J_{t\symbol{94}2}\right) -c\left( J_{t\symbol{94}1}\right) =\frac{1}{2}%
\left\vert J_{t}\right\vert -\frac{1}{2}\left\vert J_{t\symbol{94}%
i}\right\vert =d_{r}-d_{r+1}\text{.}
\end{equation*}%
Hence by induction we get $c\left( J_{s}\right) =\sum_{r=1}^{n}s_{r}\cdot
\left( d_{r-1}-d_{r}\right) $, which gives (\ref{1-2a}). From the equality 
\begin{equation*}
C_{n}(\lambda )-C_{n}(\lambda )=\bigcup_{p\in
\{0,1\}^{n}}I_{p}-\bigcup_{q\in \{0,1\}^{n}}I_{q}=\bigcup_{p,q\in
\{0,1\}^{n}}(I_{p}-I_{q})=\bigcup_{t\in \{0,1,2\}^{n}}J_{t}
\end{equation*}%
we obtain (\ref{1-5}). If $\lambda _{n+1}\geq \frac{1}{3}$, then by (\ref%
{1-4}), we have $J_{s}=J_{s\symbol{94}0}\cup J_{s\symbol{94}1}\cup J_{s%
\symbol{94}2}$. Consequently, $J_{s}=\bigcup_{t\in \left\{ 0,1,2\right\}
^{n+1},s\prec t}J_{t}$, which implies (\ref{1-6}). To prove (\ref{1-7}) it
is enough to show that for any nonincreasing sequences $(A_{n})$ and $%
(B_{n}) $ of compact subsets of $\mathbb{R}$,%
\begin{equation*}
\bigcap_{n=1}^{\infty }A_{n}-\bigcap_{n=1}^{\infty
}B_{n}=\bigcap_{n=1}^{\infty }(A_{n}-B_{n}).
\end{equation*}%
The inclusion "$\subset $" is clear. Let $x\in \bigcap_{n=1}^{\infty
}(A_{n}-B_{n}).$ Then for any $n\in \mathbb{N}$ we have $x=\alpha _{n}-\beta
_{n}$, where $\alpha _{n}\in A_{n}$ and $\beta _{n}\in B_{n}$. Pick
subsequences $(\alpha _{k_{n}})$ and $(\beta _{k_{n}})$ convergent to $%
\alpha \in A_{1}$ and $\beta \in B_{1}$, respectively. Since sequences $%
(A_{n})$ and $(B_{n})$ are nonincreasing, we have that $\alpha \in
\bigcap_{n=1}^{\infty }A_{n}$ and $\beta \in \bigcap_{n=1}^{\infty }B_{n}$.
Hence 
\begin{equation*}
x=\lim\limits_{k\rightarrow \infty }(\alpha _{n_{k}}-\beta _{n_{k}})=\alpha
-\beta \in \bigcap_{n=1}^{\infty }A_{n}-\bigcap_{n=1}^{\infty }B_{n},
\end{equation*}%
which ends the proof of (\ref{1-7}). The sets $C_{n}(\lambda )$ are
symmetric with respect to $\frac{1}{2}$. Therefore, $C(\lambda )$ is also
symmetric with respect to $\frac{1}{2}$, which implies (\ref{1-8}).
\end{proof}
\end{proposition}

A set is said to be \emph{regular closed} if it is equal to the closure of
its interior. A non-empty regular closed bounded subset $C$ of $\mathbb{R}$
is called a \emph{Cantorval}, if for every connected component $I$ of $C$
with nonempty interior, the both endpoints of $I$ are accumulation points of
the union of one-point components of $C$ (see \cite{BBFS}, \cite{BGM}).

\section{Difference sets of central Cantor sets}

One of the first results concerning the difference sets of central Cantor sets was proved by Kraft in \cite{K}. He showed that, if a sequence $\lambda =\left( \lambda _{n}\right) \in \left( 0,%
\frac{1}{2}\right) ^{\mathbb{N}}\ $ is constant, that is, $\lambda _{n}=\alpha 
$ for any $n\in \mathbb{N}$, then $C\left(
\lambda \right) -C\left( \lambda \right) = [-1,1]$ if and only if $\alpha \geq \frac{1}{3}$. Moreover, if $\alpha < \frac{1}{3}$, then $C\left(
\lambda \right) -C\left( \lambda \right)$ is a Cantor set.
Anisca and Ilie proved an important trichotomy theorem on finite sums of central Cantor sets in \cite[Theorem 2]{AI}.
Below we present a particular version of their result, which corresponds to our considerations.
\begin{theorem}[\protect\cite{AI}]
\label{AI} For any sequence $\lambda \in \left( 0,\frac{1}{2}\right) ^{%
\mathbb{N}}$, the set $C\left( \lambda \right) -C\left( \lambda \right) $
has one of the following fashions:

\begin{enumerate}
\item a finite union of closed intervals;

\item a Cantor set;

\item a Cantorval.
\end{enumerate}
\end{theorem}

We will begin our discussion on the difference sets of central Cantor sets by
proving that $C\left( \lambda \right) -C\left( \lambda \right) $ is a finite
union of intervals if and only if $\lambda _{n}\geq \frac{1}{3}$ for almost
all $n$'s. Similar results, with proofs based on other tools, were
shown by Anisca and Ilie (see \cite[Theorem 3]{AI}).

\begin{theorem}
\label{tw1} Let $\lambda =\left( \lambda _{n}\right) \in \left( 0,\frac{1}{2}%
\right) ^{\mathbb{N}}$. The following statements hold.

\begin{enumerate}
\item \label{tw1-1}$C\left( \lambda \right) -C\left( \lambda \right) =\left[
-1,1\right] $ if and only if $\lambda _{n}\geq \frac{1}{3}$ for all $n\in 
\mathbb{N}$.

\item \label{tw1-2}$C\left( \lambda \right) -C\left( \lambda \right) $ is a
finite union of intervals if and only if the set $\left\{ n\in \mathbb{N}%
:\lambda _{n}<\frac{1}{3}\right\} $ is finite.
\end{enumerate}
\end{theorem}

\begin{proof}
We will only prove (\ref{tw1-2}) (the proof of (\ref{tw1-1}) is similar).

"$\Leftarrow $" Let $k\in \mathbb{N}$ be such that $\lambda _{n}\geq \frac{1%
}{3}$ for all $n\geq k$. From Proposition \ref{lem} it follows that $%
C_{n+1}\left( \lambda \right) -C_{n+1}\left( \lambda \right) =C_{n}\left(
\lambda \right) -C_{n}\left( \lambda \right) $ for $n\geq k$, and
consequently 
\begin{equation*}
C\left( \lambda \right) -C\left( \lambda \right) =C_{k}\left( \lambda
\right) -C_{k}\left( \lambda \right) =\bigcup_{s\in \left\{ 0,1,2\right\}
^{k}}J_{s}.
\end{equation*}

"$\Rightarrow $" Assume that $C\left( \lambda \right) -C\left( \lambda
\right) $ is a finite union of closed intervals, but $\left\{ n\in \mathbb{N}%
:\lambda _{n}<\frac{1}{3}\right\} $ is infinite. Then there are $w>0$ and $%
n\in \mathbb{N}$ such that $\left[ -1,-1+w\right] \subset C\left( \lambda
\right) -C\left( \lambda \right) $, $2d_{n-1}<w$, and $\lambda _{n}<\frac{1}{3%
}$. Let $x\in G_{0^{\left( n-1\right) }}^{0}$. Since 
\begin{equation*}
x\in J_{0^{\left( n-1\right) }}=\left[ -1,-1+2d_{n-1}\right] \subset \left[
-1,-1+w\right] \subset C\left( \lambda \right) -C\left( \lambda \right)
\subset C_{n}\left( \lambda \right) -C_{n}\left( \lambda \right) ,
\end{equation*}%
there is $s\in \left\{ 0,1,2\right\} ^{n}$ such that $x\in J_{s}$. The
sequence $s$ is not an extension of $0^{\left( n-1\right) }$ because $x\in
G_{0^{\left( n-1\right) }}^{0}$. Hence $s_{k}>0$ for some $k<n$, and
consequently 
\begin{eqnarray*}
l\left( J_{s}\right) &\leq &x<r\left( G_{0^{\left( n-1\right) }}^{0}\right)
=l\left( J_{0^{\left( n-1\right) }\symbol{94}1}\right) =-1+d_{n-1}-d_{n} \\
&\leq &-1+d_{k-1}-d_{k}=l\left( J_{0^{\left( k-1\right) }\symbol{94}%
1}\right) \leq l\left( J_{s}\right) ,
\end{eqnarray*}%
a contradiction.
\end{proof}

Our goal is to find sufficient conditions for the set $C\left( \lambda
\right) -C\left( \lambda \right) $ to be a Cantorval. From Theorem \ref{tw1}
it follows that we need to assume that $\lambda _{n}<\frac{1}{3}$ for
infinitely many terms. From Theorem \ref{AI} we infer that, for a sequence $%
\lambda $ satisfying the above condition, it suffices to prove that $C\left(
\lambda \right) -C\left( \lambda \right) $ has nonempty interior. To show
this, we will take an interval $J_{t}$ and define the family of gaps $%
\mathcal{G}_{t}$ such that $J_{t}\setminus \bigcup \mathcal{G}_{t}\subset
C\left( \lambda \right) -C\left( \lambda \right) $ and $\text{int}\left(
J_{t}\setminus \bigcup \mathcal{G}_{t}\right) \neq \emptyset $. The proof of
the second condition requires the assumption that $\lambda _{n}\geq \frac{1}{%
3}$ for infinitely many terms (Theorem \ref{tw4}). To prove the first
condition, we will need far stronger assumptions (Theorem \ref{Th1}).

We now introduce additional notation and definitions used in the next
theorems. We consider sequences such that $\lambda _{n}<\frac{1}{3}$ for
infinitely many terms. For a fixed sequence $\lambda $ (and a fixed number $%
k_{0}$) we need the sequence $\left( k_{n}\right) $ consisting of all
indices greater than $k_{0}$, for which $\lambda _{k_{n}}<\frac{1}{3}$, and
the families $\mathcal{G}_{t}\left( n\right) $ of some particular gaps in intervals $J_{s}$
for $s\in \left\{ 0,1,2\right\} ^{k_{n}-1}$, $t \prec s$.

Let $\lambda =\left( \lambda _{j}\right) _{j\in \mathbb{N}}\in \left( 0,%
\frac{1}{2}\right) ^{\mathbb{N}}$ be such that $\lambda _{j}<\frac{1}{3}$
for infinitely many terms and $\lambda _{k_{0}+1}>\frac{1}{3}$ for some $%
k_{0}\in \mathbb{N}\cup \left\{ 0\right\} $. By $\left( k_{n}\right) _{n\in 
\mathbb{N}}$ we denote an increasing sequence such that $\left( \lambda
_{k_{n}}\right) _{n\in \mathbb{N}}$ is a subsequence of the sequence $\left(
\lambda _{j}\right) _{j>k_{0}}$ consisting of all terms which are less than $%
\frac{1}{3}$.

Fix $k\geq k_{0}$ and $t\in \left\{ 0,1,2\right\} ^{k}$. There is a unique $%
m\in \mathbb{N}$ such that $k_{m-1}\leq k<k_{m}$. For $n\geq m$ we define
families $\mathcal{G}_{t}\left( n\right) \subset \left\{ G_{s}^{i}:s\in
\left\{ 0,1,2\right\} ^{k_{n}-1},i\in \left\{ 0,1\right\} \right\} $
inductively as follows.

First, let $\widetilde{\mathcal{G}_{t}}\left( n\right) :=\left\{ G_{t\symbol{%
94}0^{\left( k_{n}-k-1\right) }}^{0},G_{t\symbol{94}2^{\left(
k_{n}-k-1\right) }}^{1}\right\} $ for $n\geq m$. Then put $\mathcal{G}%
_{t}\left( m\right) :=\widetilde{\mathcal{G}_{t}}\left( m\right) $. Assume
that we have defined families $\mathcal{G}_{t}\left( l\right) $ for $m\leq
l\leq n$. Then define $\mathcal{G}_{t}\left( n+1\right) $ as 
\begin{equation*}
\widetilde{\mathcal{G}_{t}}\left( n+1\right) \cup \bigcup_{l=m}^{n}\left(
\left\{ G_{s\symbol{94}\left( i+1\right) \symbol{94}0^{\left(
k_{n+1}-k_{l}-1\right) }}^{0}:G_{s}^{i}\in \mathcal{G}_{t}\left( l\right)
\right\} \cup \left\{ G_{s\symbol{94}i\symbol{94}2^{\left(
k_{n+1}-k_{l}-1\right) }}^{1}:G_{s}^{i}\in \mathcal{G}_{t}\left( l\right)
\right\} \right) .
\end{equation*}%
Geometrically, the family $\widetilde{\mathcal{G}_{t}}\left( n\right) $
consists of the leftmost gap and the rightmost gap in the interval $J_{t}$ which
appear in $C_{k_{n}}\left( \lambda \right) -C_{k_{n}}\left( \lambda \right) $%
. The family $\mathcal{G}_{t}\left( n\right) $ consists of the gaps from $%
\widetilde{\mathcal{G}_{t}}\left( n\right) $ and the gaps that appear in $%
C_{k_{n}}\left( \lambda \right) -C_{k_{n}}\left( \lambda \right) $ which are
nearest the gaps from $\bigcup_{l=m}^{n-1}\mathcal{G}_{t}\left( l\right) $.

Moreover, let $\mathcal{G}_{t}:=\bigcup_{n\geq m}\mathcal{G}_{t}\left(
n\right) $ and%
\begin{eqnarray*}
p\left( n\right) \left. :=\right. t\symbol{94}0^{\left( k_{m}-k-1\right) }%
\symbol{94}1\symbol{94}0^{\left( k_{m+1}-k_{m}-1\right) }\symbol{94}1\symbol{%
94}\ldots \symbol{94}1\symbol{94}0^{\left( k_{n}-k_{n-1}-1\right) } &\in
&\left\{ 0,1,2\right\} ^{k_{n}-1}, \\
q\left( n\right) \left. :=\right. t\symbol{94}2^{\left( k_{m}-k-1\right) }%
\symbol{94}1\symbol{94}2^{\left( k_{m+1}-k_{m}-1\right) }\symbol{94}1\symbol{%
94}\ldots \symbol{94}1\symbol{94}2^{\left( k_{n}-k_{n-1}-1\right) } &\in
&\left\{ 0,1,2\right\} ^{k_{n}-1}.
\end{eqnarray*}

\begin{lemma}
\label{lem2}Assume that $\lambda =\left( \lambda _{j}\right) _{j\in \mathbb{N%
}}\in \left( 0,\frac{1}{2}\right) ^{\mathbb{N}}$ is a sequence such that: $%
\lambda _{n}<\frac{1}{3}$\ for infinitely many terms, $\lambda _{n}\geq 
\frac{1}{3}$\ for infinitely many terms, and there is $k_{0}\in \mathbb{N}%
\cup \left\{ 0\right\} $ such that $\lambda _{k_{0}+1}>\frac{1}{3}$. Let $%
k\geq k_{0}$, $t\in \left\{ 0,1,2\right\} ^{k}$, and $m\in \mathbb{N}$ be
such that $k_{m-1}\leq k<k_{m}$. The following statements hold.

\begin{enumerate}
\item \label{lem2-1}The sequence $r\left( G_{p\left( n\right) }^{0}\right) $
is nondecreasing, the sequence $l\left( G_{q\left( n\right) }^{1}\right) $
is nonincreasing, 
\begin{equation*}
\lim_{n\rightarrow \infty }r\left( G_{p\left( n\right) }^{0}\right)
<\lim_{n\rightarrow \infty }l\left( G_{q\left( n\right) }^{1}\right) ,
\end{equation*}%
and for any $n\geq m$, we have%
\begin{equation*}
r\left( G_{p\left( n\right) }^{0}\right) =l\left( J_{t}\right)
+\sum_{l=m}^{n}\left( d_{k_{l}-1}-d_{k_{l}}\right) \quad \text{and}\quad
l\left( G_{q\left( n\right) }^{1}\right) =r\left( J_{t}\right)
-\sum_{l=m}^{n}\left( d_{k_{l}-1}-d_{k_{l}}\right) .
\end{equation*}

\item \label{lem2-2}If $n\geq m$ and $G\in \mathcal{G}_{t}\left( n\right) $,
then 
\begin{equation*}
r\left( G\right) \leq r\left( G_{p\left( n\right) }^{0}\right) \text{\quad
or\quad }l\left( G\right) \geq l\left( G_{q\left( n\right) }^{1}\right) .
\end{equation*}

\item \label{lem2-0}If $n\geq M>m$, $k_{M-1}\leq r<k_{M}$, $w\in \left\{
0,1,2\right\} ^{r}$, and $t\prec w$, then 
\begin{equation*}
\mathcal{G}_{t}\left( n\right) \cap \left\{ G_{s}^{i}:s\in \left\{
0,1,2\right\} ^{k_{n}-1},w\prec s,i\in \left\{ 0,1\right\} \right\} \subset 
\mathcal{G}_{w}\left( n\right) .
\end{equation*}

\item \label{lem2-4}If $k=k_{0}=0$, $t=\emptyset $, $n\in \N$, $s\in \left\{
0,1,2\right\} ^{k_{n}-1}$, $i\in \left\{ 0,1\right\} $, and $G_{s}^{i}\in 
\mathcal{G}_{t}$, then 
\begin{gather*}
\left. G_{s}^{i}\cap J_{v}=\emptyset \right. \text{ if }v\in \left\{
0,1,2\right\} ^{k_{n}-1},v\neq s, \\
\left. G_{s}^{i}\cap G_{u}^{j}=\emptyset \right. \text{ if }G_{u}^{j}\in 
\mathcal{G}_{t},\left( u,j\right) \neq \left( s,i\right) .
\end{gather*}
\end{enumerate}

\begin{proof}
Ad (\ref{lem2-1}). From Proposition \ref{lem}(\ref{1-2a}) we have%
\begin{equation*}
l\left( J_{p\left( n\right) \symbol{94}1}\right) -l\left( J_{t}\right)
=l\left( J_{p\left( n\right) \symbol{94}1}\right) -l\left( J_{t\symbol{94}%
0^{\left( k_{n}-k\right) }}\right) =\sum_{l=m}^{n}\left(
d_{k_{l}-1}-d_{k_{l}}\right) ,
\end{equation*}%
and therefore $r\left( G_{p\left( n\right) }^{0}\right) =l\left( J_{p\left(
n\right) \symbol{94}1}\right) =l\left( J_{t}\right) +\sum_{l=m}^{n}\left(
d_{k_{l}-1}-d_{k_{l}}\right) $. Hence the sequence $r\left( G_{p\left(
n\right) }^{0}\right) $ is nondecreasing. Similarly, we prove that $l\left(
G_{q\left( n\right) }^{1}\right) =r\left( J_{t}\right) -\sum_{l=m}^{n}\left(
d_{k_{l}-1}-d_{k_{l}}\right) $ and the sequence $l\left( G_{q\left( n\right)
}^{1}\right) $ is nonincreasing. Since the set $\mathbb{N}\setminus \left\{
k_{n}:n\in \mathbb{N}\right\} $ is infinite, we have%
\begin{eqnarray*}
\lim_{n\rightarrow \infty }l\left( G_{q\left( n\right) }^{1}\right)
-\lim_{n\rightarrow \infty }r\left( G_{p\left( n\right) }^{0}\right)
&=&r\left( J_{t}\right) -l\left( J_{t}\right) -2\sum_{l=m}^{\infty }\left(
d_{k_{l}-1}-d_{k_{l}}\right) \\
&>&\left\vert J_{t}\right\vert -2\sum_{r=k}^{\infty }\left(
d_{r}-d_{r+1}\right) =\left\vert J_{t}\right\vert -2d_{k}=0.
\end{eqnarray*}

Ad (\ref{lem2-2}). We prove it inductively. Observe that $p\left( m\right) =t%
\symbol{94}0^{\left( k_{m}-k-1\right) }$, $q\left( m\right) =t\symbol{94}%
2^{\left( k_{m}-k-1\right) }$, and $\mathcal{G}_{t}\left( m\right) =\left\{
G_{p\left( m\right) }^{0},G_{q\left( m\right) }^{1}\right\} $. Thus,
condition (\ref{lem2-2}) holds for $n=m$. Assume that $n\geq m$ and for any $%
l\in \left\{ m,\ldots ,n\right\} $ and $G\in \mathcal{G}_{t}\left( l\right) $
we have $r\left( G\right) \leq r\left( G_{p\left( l\right) }^{0}\right) \ $%
or\ $l\left( G\right) \geq l\left( G_{q\left( l\right) }^{1}\right) $. Let $%
G\in \mathcal{G}_{t}\left( n+1\right) $. We will prove that $r\left(
G\right) \leq r\left( G_{p\left( n+1\right) }^{0}\right) \ $or\ $l\left(
G\right) \geq l\left( G_{q\left( n+1\right) }^{1}\right) $. Let us consider
the cases.

1$^{\text{o}}$ $G=G_{t\symbol{94}0^{\left( k_{n+1}-k-1\right) }}^{0}$ or $%
G=G_{t\symbol{94}2^{\left( k_{n+1}-k-1\right) }}^{1}$. \newline
The inequalities $r\left( G_{t\symbol{94}0^{\left( k_{n+1}-k-1\right)
}}^{0}\right) <r\left( G_{p\left( n+1\right) }^{0}\right) $ and $l\left( G_{t%
\symbol{94}2^{\left( k_{n+1}-k-1\right) }}^{1}\right) >l\left( G_{q\left(
n+1\right) }^{1}\right) $ are obvious.

2$^{\text{o}}$ $G=G_{s\symbol{94}\left( i+1\right) \symbol{94}0^{\left(
k_{n+1}-k_{l}-1\right) }}^{0}$, where $l\in \left\{ m,\ldots ,n\right\} $, $%
i\in \left\{ 0,1\right\} $, and $G_{s}^{i}\in \mathcal{G}_{t}\left( l\right) $%
. \newline
By the inductive hypothesis and (\ref{lem2-1}) we get $r\left(
G_{s}^{i}\right) \leq r\left( G_{p\left( l\right) }^{0}\right) \leq r\left(
G_{p\left( n\right) }^{0}\right) $ or $l\left( G_{s}^{i}\right) \geq l\left(
G_{q\left( l\right) }^{1}\right) \geq l\left( G_{q\left( n\right)
}^{1}\right) $. In the first case, we have 
\begin{eqnarray*}
r\left( G\right)  &=&r\left( G_{s\symbol{94}\left( i+1\right) \symbol{94}%
0^{\left( k_{n+1}-k_{l}-1\right) }}^{0}\right) =l\left( J_{s\symbol{94}%
\left( i+1\right) \symbol{94}0^{\left( k_{n+1}-k_{l}-1\right) }}\right)
+\left( d_{k_{n+1}-1}-d_{k_{n+1}}\right)  \\
&=&r\left( G_{s}^{i}\right) +\left( d_{k_{n+1}-1}-d_{k_{n+1}}\right) \leq
r\left( G_{p\left( n\right) }^{0}\right) +\left(
d_{k_{n+1}-1}-d_{k_{n+1}}\right) =r\left( G_{p\left( n+1\right) }^{0}\right)
,
\end{eqnarray*}%
and in the second case,
\begin{equation*}
l\left( G\right) >l\left( J_{s\symbol{94}\left( i+1\right) }\right) >l\left(
G_{s}^{i}\right) \geq l\left( G_{q\left( n\right) }^{1}\right) >l\left(
G_{q\left( n+1\right) }^{1}\right) .
\end{equation*}

3$^{\text{o}}$ $G=G_{s\symbol{94}i\symbol{94}2^{\left(
k_{n+1}-k_{l}-1\right) }}^{1}$, where $l\in \left\{ m,\ldots ,n\right\} $, $%
i\in \left\{ 0,1\right\} $, and $G_{s}^{i}\in \mathcal{G}_{t}\left( l\right) $%
. \newline
The proof is analogous to the proof in the previous case.

Ad (\ref{lem2-0}). Fix $M$, $r$, $w$ such that $M>m$, $k_{M-1}\leq r<k_{M}$, 
$w\in \left\{ 0,1,2\right\} ^{r}$, $t\prec w$. Put 
\begin{equation*}
\mathcal{H}\left( n\right) :=\mathcal{G}_{t}\left( n\right) \cap \left\{
G_{s}^{i}:s\in \left\{ 0,1,2\right\} ^{k_{n}-1},w\prec s,i\in \left\{
0,1\right\} \right\} .
\end{equation*}%
We will prove inductively that $\mathcal{H}\left( n\right) \subset \mathcal{G%
}_{w}\left( n\right) $ for $n\geq M$. First, we will show that $\mathcal{H}%
\left( M\right) \subset \mathcal{G}_{w}\left( M\right) $. Let $G\in \mathcal{%
H}\left( M\right) $, that is $G=G_{s}^{i}\in \mathcal{G}_{t}\left( M\right) $%
, where $s\in \left\{ 0,1,2\right\} ^{k_{M}-1}$, $w\prec s$, $i\in \left\{
0,1\right\} $. Let us consider the cases.

1$^{\text{o}}$ $G=G_{s}^{0}$, where $s=t\symbol{94}0^{\left(
k_{M}-k-1\right) }$ or $G=G_{s}^{1}$, where $s=t\symbol{94}2^{\left(
k_{M}-k-1\right) }$.\newline
In the former case, $s=w\symbol{94}0^{\left( k_{M}-r-1\right) }$, so $G=G_{w%
\symbol{94}0^{\left( k_{M}-r-1\right) }}^{0}\in \mathcal{G}_{w}\left(
M\right) $, and in the other case, $s=w\symbol{94}2^{\left( k_{M}-r-1\right)
} $, so $G=G_{w\symbol{94}2^{\left( k_{M}-r-1\right) }}^{1}\in \mathcal{G}%
_{w}\left( M\right) $.

2$^{\text{o}}$ $G=G_{s}^{0}$, where $s=v\symbol{94}\left( i+1\right) \symbol{%
94}0^{\left( k_{M}-k_{l}-1\right) }$, $i\in \left\{ 0,1\right\} $, $m\leq
l<M $, $v\in \left\{ 0,1,2\right\} ^{k_{l}-1}$, $G_{v}^{i}\in \mathcal{G}%
_{t}\left( l\right) $.\newline
Since $k_{l}\leq k_{M-1}\leq r$ and $w\prec s$, we have $s=w\symbol{94}%
0^{\left( k_{M}-r-1\right) }$, and so $G=G_{w\symbol{94}0^{\left(
k_{M}-r-1\right) }}^{0}\in \mathcal{G}_{w}\left( M\right) $.

3$^{\text{o}}$ $G=G_{s}^{1}$, where $s=v\symbol{94}i\symbol{94}2^{\left(
k_{M}-k_{l}-1\right) }$, $i\in \left\{ 0,1\right\} $, $m\leq l<M$, $v\in
\left\{ 0,1,2\right\} ^{k_{l}-1}$, $G_{v}^{i}\in \mathcal{G}_{t}\left(
l\right) $.\newline
As above we get $s=w\symbol{94}2^{\left( k_{M}-r-1\right) }$, so $G=G_{w%
\symbol{94}2^{\left( k_{M}-r-1\right) }}^{1}\in \mathcal{G}_{w}\left(
M\right) $. This completes the proof of the inclusion $\mathcal{H}\left(
M\right) \subset \mathcal{G}_{w}\left( M\right) $.

Assume now that $n\geq M$ and $\mathcal{H}\left( l\right) \subset \mathcal{G}%
_{w}\left( l\right) $ for $l\in \left\{ M,\ldots ,n\right\} $. Let $G\in 
\mathcal{H}\left( n+1\right) $, that is, $G=G_{s}^{i}\in \mathcal{G}%
_{t}\left( n+1\right) $, where $s\in \left\{ 0,1,2\right\} ^{k_{n+1}-1}$, $%
w\prec s$, $i\in \left\{ 0,1\right\} $. We will prove that $G\in \mathcal{G}%
_{w}\left( n+1\right) $. Let us consider the cases.

1$^{\text{o}}$ $G=G_{s}^{0}$, where $s=t\symbol{94}0^{\left(
k_{n+1}-k-1\right) }$ or $G=G_{s}^{1}$, where $s=t\symbol{94}2^{\left(
k_{n+1}-k-1\right) }$.\newline
In the first case, $s=w\symbol{94}0^{\left( k_{n+1}-r-1\right) }$, so $G=G_{w%
\symbol{94}0^{\left( k_{n+1}-r-1\right) }}^{0}\in \mathcal{G}_{w}\left(
n+1\right) $, and in the second case, $s=w\symbol{94}2^{\left(
k_{n+1}-r-1\right) }$, so $G=G_{w\symbol{94}2^{\left( k_{n+1}-r-1\right)
}}^{1}\in \mathcal{G}_{w}\left( n+1\right) $.

2$^{\text{o}}$ $G=G_{s}^{0}$, where $s=v\symbol{94}\left( i+1\right) \symbol{%
94}0^{\left( k_{n+1}-k_{l}-1\right) }$, $i\in \left\{ 0,1\right\} $, $m\leq
l\leq n$, $v\in \left\{ 0,1,2\right\} ^{k_{l}-1}$, $G_{v}^{i}\in \mathcal{G}%
_{t}\left( l\right) $.\newline
If $k_{l}\leq r$, then $s=w\symbol{94}0^{\left( k_{n+1}-r-1\right) }$,
because $w\prec s$. Hence $G=G_{w\symbol{94}0^{\left( k_{n+1}-r-1\right)
}}^{0}\in \mathcal{G}_{w}\left( n+1\right) $. If $k_{l}>r$, then $w\prec v$,
and therefore $G_{v}^{i}\in \mathcal{H}\left( l\right) $. From the inductive
hypothesis it follows that $G_{v}^{i}\in \mathcal{G}_{w}\left( l\right) $,
and consequently $G=G_{v\symbol{94}\left( i+1\right) \symbol{94}0^{\left(
k_{n+1}-k_{l}-1\right) }}^{0}\in \mathcal{G}_{w}\left( n+1\right) $.

3$^{\text{o}}$ $G=G_{s}^{1}$, where $s=v\symbol{94}i\symbol{94}2^{\left(
k_{n+1}-k_{l}-1\right) }$, $i\in \left\{ 0,1\right\} $, $m\leq l\leq n$, $%
v\in \left\{ 0,1,2\right\} ^{k_{l}-1}$, $G_{v}^{i}\in \mathcal{G}_{t}\left(
l\right) $.\newline
Similarly as in the previous case, we show that $G\in \mathcal{G}_{w}\left(
n+1\right) $, which finishes the proof of (\ref{lem2-0}).

Ad (\ref{lem2-4}). Assume that $k=k_{0}=0$, $t=\emptyset $ and set $\mathcal{%
G}:=\mathcal{G}_{\emptyset }$. Let $n \in \N$, $s\in \left\{ 0,1,2\right\}
^{k_{n}-1}$, $i\in \left\{ 0,1\right\} $ and $G_{s}^{i}\in \mathcal{G}_{t}$.
We first prove that for any $v\in \left\{ 0,1,2\right\} ^{k_{n}-1}$, $v\neq
s $ we have $G_{s}^{i}\cap J_{v}=\emptyset $. Let us consider the cases.

1$^{\text{o}}$ There is $l<n$ such that $s_{j}=v_{j}$ for $j=1,\ldots
,k_{l}-1$ and $s_{k_{l}}\neq v_{k_{l}}$. \newline
Writing $p:=s|\left( k_{l}-1\right) $, we get $G_{s}^{i}\subset J_{s}\subset
J_{s|k_{l}}=J_{p\symbol{94}s_{k_{l}}}$ and $J_{v}\subset J_{v|k_{l}}=J_{p%
\symbol{94}v_{k_{l}}}$. Since $\lambda _{k_{l}}<\frac{1}{3}$, the intervals $%
J_{p\symbol{94}s_{k_{l}}}$ and $J_{p\symbol{94}v_{k_{l}}}$ are disjoint.
Hence $G_{s}^{i}\cap J_{v}=\emptyset $.

2$^{\text{o}}$ There is $r\leq k_{n}-1$ such that $s_{j}=v_{j}$ for $%
j=1,\ldots ,r-1$ and $\left\vert s_{r}-v_{r}\right\vert =2$. \newline
Taking $p:=s|\left( r-1\right) $, we obtain $G_{s}^{i}\subset J_{s}\subset
J_{s|r}=J_{p\symbol{94}s_{r}}$ and $J_{v}\subset J_{v|r}=J_{p\symbol{94}%
v_{r}}$. Since $\lambda _{r}<\frac{1}{2}$, the intervals $J_{p\symbol{94}%
s_{r}}$ and $J_{p\symbol{94}v_{r}}$ are disjoint. Therefore, $G_{s}^{i}\cap
J_{v}=\emptyset .$

3$^{\text{o}}$ There are $l$, $r$ such that $l\leq n$, $k_{l-1}<r<k_{l}$, $%
s_{j}=v_{j}$ for $j=1,\ldots ,r-1$ and $v_{r}=s_{r}+1$.\newline
Put $p:=s|\left( r-1\right) $ and 
\begin{eqnarray*}
\overline{p}\left. :=\right. p\symbol{94}0^{\left( k_{l}-r\right) }\symbol{94%
}1\symbol{94}0^{\left( k_{l+1}-k_{l}-1\right) }\symbol{94}1\symbol{94}\ldots 
\symbol{94}1\symbol{94}0^{\left( k_{n}-k_{n-1}-1\right) } &\in &\left\{
0,1,2\right\} ^{k_{n}-1}, \\
\overline{q}\left. :=\right. p\symbol{94}2^{\left( k_{l}-r\right) }\symbol{94%
}1\symbol{94}2^{\left( k_{l+1}-k_{l}-1\right) }\symbol{94}1\symbol{94}\ldots 
\symbol{94}1\symbol{94}2^{\left( k_{n}-k_{n-1}-1\right) } &\in &\left\{
0,1,2\right\} ^{k_{n}-1}.
\end{eqnarray*}%
From (\ref{lem2-0}) it follows that $G_{s}^{i}\in \mathcal{G}_{p}\left(
n\right) $. Consequently, (\ref{lem2-2}) shows that $r\left(
G_{s}^{i}\right) \leq r\left( G_{\overline{p}}^{0}\right) $ or $l\left(
G_{s}^{i}\right) \geq l\left( G_{\overline{q}}^{1}\right) $.\ Since $s_{r}<2$%
, by the definition of $\mathcal{G}$ we get $s_{r}=\ldots =s_{k_{l}-1}=0$.
Thus,
\begin{equation*}
l\left( G_{s}^{i}\right) <r\left( J_{s}\right) \leq r\left( J_{s|r}\right)
=r\left( J_{p\symbol{94}0}\right) <l\left( J_{p\symbol{94}2}\right) =l\left(
J_{\overline{q}|r}\right) \leq l\left( J_{\overline{q}}\right) <l\left( G_{%
\overline{q}}^{1}\right) ,
\end{equation*}%
and therefore $r\left( G_{s}^{i}\right) \leq r\left( G_{\overline{p}%
}^{0}\right) $. On the other hand, from (\ref{lem2-1}) it follows that%
\begin{eqnarray*}
r\left( G_{\overline{p}}^{0}\right) &=&l\left( J_{p}\right)
+\sum_{j=l}^{n}\left( d_{k_{j}-1}-d_{k_{j}}\right) \leq l\left( J_{p}\right)
+\sum_{j=r+1}^{k_{n}}\left( d_{j-1}-d_{j}\right) <l\left( J_{p\symbol{94}%
s_{r}}\right) +d_{r} \\
&<&l\left( J_{p\symbol{94}s_{r}}\right) +d_{r-1}-d_{r}=l\left( J_{p\symbol{94%
}\left( s_{r}+1\right) }\right) =l\left( J_{v|r}\right) \leq l\left(
J_{v}\right) .
\end{eqnarray*}%
Hence $r\left( G_{s}^{i}\right) <l\left( J_{v}\right) $, and so $%
G_{s}^{i}\cap J_{v}=\emptyset $.

4$^{\text{o}}$ There are $l$, $r$ such that $l\leq n$, $k_{l-1}<r<k_{l}$, $%
s_{j}=v_{j}$ for $j=1,\ldots ,r-1$ and $v_{r}=s_{r}-1$. \newline
The reasoning is analogous to that in the previous case. This finishes the
proof of the first condition in (\ref{lem2-4}).

Assume now that $G_{u}^{j}\in \mathcal{G}_{t}$, where $u\in \left\{
0,1,2\right\} ^{l}$, $j\in \left\{ 0,1\right\} $, and $\left( u,j\right) \neq
\left( s,i\right) $. If $s=u$ and $i\neq j$, then the condition $%
G_{s}^{i}\cap G_{u}^{j}=\emptyset $ is obvious. Suppose that $s\neq u$.
Without loss of generality we can assume that $n<l$. Thus, we have $%
G_{u}^{j}\subset J_{u}\subset J_{u|k_{n}}\subset J_{u|\left( k_{n}-1\right)
} $. If $s=u|\left( k_{n}-1\right) $, then $G_{s}^{i}\cap
J_{u|k_{n}}=G_{s}^{i}\cap J_{s\symbol{94}u_{k_{n}}}=\emptyset $, and
therefore $G_{s}^{i}\cap G_{u}^{j}=\emptyset $. If $s\neq u|\left(
k_{n}-1\right) $, then from the first condition we get $G_{s}^{i}\cap
J_{u|\left( k_{n}-1\right) }=\emptyset $, which gives $G_{s}^{i}\cap
G_{u}^{j}=\emptyset $.
\end{proof}
\end{lemma}

The next theorem yields information about relationships between the sets $%
C\left( \lambda \right) -C\left( \lambda \right) $ and $\bigcup \mathcal{G}%
_{t}$. It states that, under the same assumptions as in Lemma \ref{lem2}, $%
\text{int}\left( J_{t}\setminus \bigcup \mathcal{G}_{t}\right) \neq
\emptyset $. Moreover, this theorem shows that for $t=\emptyset $ the
inclusion $C\left( \lambda \right) -C\left( \lambda \right) \subset \left[
-1,1\right] \setminus \bigcup \mathcal{G}_{t}$ holds, and it gives the
formula for a measure of the set $\bigcup \mathcal{G}_{t}$. Using these
conditions, in Theorem \ref{Th1} we will prove (under additional
assumptions) that $\text{int}\left( C\left( \lambda \right) -C\left( \lambda
\right) \right) \neq \emptyset $ and, in consequence, $C\left( \lambda
\right) -C\left( \lambda \right) $ is a Cantorval. We will also get the
formula for the measure of the set $C\left( \lambda \right) -C\left( \lambda
\right) $.

\begin{theorem}
\label{tw4}Assume that $\lambda =\left( \lambda _{j}\right) _{j\in \mathbb{N}%
}\in \left( 0,\frac{1}{2}\right) ^{\mathbb{N}}$ is a sequence such that: $%
\lambda _{n}<\frac{1}{3}$\ for infinitely many terms, $\lambda _{n}\geq 
\frac{1}{3}$\ for infinitely many terms, and there exists $k_{0}\in \mathbb{N%
}\cup \left\{ 0\right\} $ such that $\lambda _{k_{0}+1}>\frac{1}{3}$, where the sequence $\left( k_n \right)$ consists of all indices greater than $k_0$, for which $\lambda_{k_n} < \frac{1}{3}$. Let $%
k\geq k_{0}$, $t\in \left\{ 0,1,2\right\} ^{k}$ and $m\in \mathbb{N}$ be
such that $k_{m-1}\leq k<k_{m}$. The following statements hold.

\begin{enumerate}
\item \label{tw4-1}The set $J_{t}\setminus \bigcup \mathcal{G}_{t}$ has
nonempty interior.

\item \label{tw4-2}If $k=k_{0}=0$ and $t=\emptyset $, then $C\left( \lambda
\right) -C\left( \lambda \right) \subset \left[ -1,1\right] \setminus
\bigcup \mathcal{G}_{t}$ and 
\begin{equation*}
\left\vert \bigcup \mathcal{G}_{t}\right\vert =\sum_{n=1}^{\infty }2\cdot
3^{n-1}\left( d_{k_{n}-1}-3d_{k_{n}}\right) .
\end{equation*}
\end{enumerate}

\begin{proof}
Put $J:=\left( \lim\limits_{n\rightarrow \infty }r\left( G_{p\left( n\right)
}^{0}\right) ,\lim\limits_{n\rightarrow \infty }l\left( G_{q\left( n\right)
}^{1}\right) \right) $. From Lemma \ref{lem2} it follows that $J\subset
J_{t} $, $\left\vert J\right\vert >0$, and for any interval $G\in \mathcal{G}%
_{t}$ we have $r\left( G\right) \leq \lim\limits_{n\rightarrow \infty
}r\left( G_{p\left( n\right) }^{0}\right) =l\left( J\right) $ or $l\left(
G\right) \geq \lim\limits_{n\rightarrow \infty }l\left( G_{q\left( n\right)
}^{1}\right) =r\left( J\right) $. Hence $G\cap J=\emptyset $ and, in
consequence, $J\cap \bigcup \mathcal{G}_{t}=\emptyset $, which completes the
proof of (\ref{tw4-1}).

Write $\mathcal{G}:=\mathcal{G}_{\emptyset }$. On the contrary, suppose that
the set $\left( C\left( \lambda \right) -C\left( \lambda \right) \right)
\cap \bigcup \mathcal{G}$ is nonempty, that is $G_{s}^{i}\cap \left( C\left(
\lambda \right) -C\left( \lambda \right) \right) \neq \emptyset $ for some $n\in \mathbb{N}$, $s\in \left\{ 0,1,2\right\} ^{k_{n}-1}$ and $i\in
\left\{ 0,1\right\} $ such that $G_{s}^{i}\in \mathcal{G}$. Hence we obtain $%
G_{s}^{i}\cap \left( C_{k_{n}}\left( \lambda \right) -C_{k_{n}}\left(
\lambda \right) \right) \neq \emptyset $, and consequently there is a
sequence $v\in \left\{ 0,1,2\right\} ^{k_{n}}$ such that $G_{s}^{i}\cap
J_{v}\neq \emptyset $. If $s\prec v$, then $v=s\symbol{94}h$, where $h\in
\left\{ 0,1,2\right\} $, and therefore $G_{s}^{i}\cap J_{v}=G_{s}^{i}\cap
J_{s\symbol{94}h}=\emptyset $, a contradiction. However, if $s\neq v|\left(
k_{n}-1\right) $, then from Lemma \ref{lem2} we get $G_{s}^{i}\cap
J_{v}\subset G_{s}^{i}\cap J_{v|\left( k_{n}-1\right) }=\emptyset $, a
contradiction. This finishes the proof of the inclusion $C\left( \lambda
\right) -C\left( \lambda \right) \subset \left[ -1,1\right] \setminus
\bigcup \mathcal{G}$.

Let $a_{n}$ denote the number of elements of the set $\mathcal{G}\left(
n\right) $. From Lemma \ref{lem2} it follows that $G_{s}^{i}\neq G_{u}^{j}$
for any $G_{s}^{i},G_{u}^{j}\in \mathcal{G}$ such that $\left( s,i\right)
\neq \left( u,j\right) $. Hence $a_{1}=2$ and $a_{n+1}=2+2%
\sum_{i=1}^{n}a_{i} $. It is easy to check that the sequence $a_{n}=2\cdot
3^{n-1}$ is a solution of this recurrence equation. Since the intervals in $%
\mathcal{G}$ are pairwise disjoint and the length of each interval from $%
\mathcal{G}\left( n\right) $ is equal to $d_{k_{n}-1}-3d_{k_{n}}$, we
have that%
\begin{equation*}
\left\vert \bigcup \mathcal{G}\right\vert =\sum_{n=1}^{\infty }\sum_{G\in 
\mathcal{G}\left( n\right) }\left\vert G\right\vert =\sum_{n=1}^{\infty
}2\cdot 3^{n-1}\left( d_{k_{n}-1}-3d_{k_{n}}\right) .
\end{equation*}
\end{proof}
\end{theorem}

The next theorem is the main result of our paper.

\begin{theorem}
\label{Th1}Assume that $\lambda =\left( \lambda _{j}\right) _{j\in \mathbb{N}%
}\in \left( 0,\frac{1}{2}\right) ^{\mathbb{N}}$ is a sequence such that: $%
\lambda _{n}<\frac{1}{3}$\ for infintely many terms, $\lambda _{n}\geq \frac{%
1}{3}$\ for infintely many terms, and there is $k_{0}\in \mathbb{N}\cup
\left\{ 0\right\} $ such that $\lambda _{k_{0}+1}>\frac{1}{3}$. Let $k \geq k_0$, $t \in \{0,1,2\}^k$ and $m \in \N$ be such that $k_{m-1} \leq k < k_m$, where the sequence $\left( k_n \right)$ consists of all indices greater than $k_0$, for which $\lambda_{k_n} < \frac{1}{3}$. Moreover,
assume that there exists a sequence $\left( \delta _{n}\right) _{n\in 
\mathbb{N}}$ such that for any $n\in \mathbb{N}$, the following conditions
hold 
\begin{equation}
\left\{ 
\begin{array}{l}
3d_{r}-d_{r-1}=\delta _{n}\text{ if }k_{n-1}<r<k_{n}, \\ 
4d_{k_{n}}=\delta _{n}+\delta _{n+1}, \\ 
d_{k_{n}-1}-d_{k_{n}}=\delta _{n}-\delta _{n+1}.%
\end{array}%
\right.  \label{07}
\end{equation}%
Then we have:

\begin{enumerate}
\item \label{Th1-1}$J_{t}\setminus \bigcup \mathcal{G}_{t}\subset C\left(
\lambda \right) -C\left( \lambda \right) $.

\item \label{Th1-2}The set $C\left( \lambda \right) -C\left( \lambda \right) 
$ is a Cantorval.

\item \label{Th1-3}If $k=k_{0}=0$ and $t=\emptyset $, then $C\left( \lambda
\right) -C\left( \lambda \right) =J_{t}\setminus \bigcup \mathcal{G}_{t}$
and 
\begin{equation*}
\left\vert C\left( \lambda \right) -C\left( \lambda \right) \right\vert
=2-2\sum_{n=1}^{\infty }3^{n-1}\left( d_{k_{n}-1}-3d_{k_{n}}\right) .
\end{equation*}
\end{enumerate}

\begin{proof}
If $n\in \mathbb{N}$, $s\in \left\{ 0,1,2\right\} ^{k_{n}-1}$, $i\in \left\{
0,1\right\} $, $u\in \left\{ 0,1,2\right\} ^{k_{n}}$, and $c\left(
G_{s}^{i}\right) =c\left( J_{u}\right) $, then we write $G_{s}^{i}\subset
_{\ast }J_{u}$. Of course, $G_{s}^{i}\subset _{\ast }J_{u}$ implies $%
G_{s}^{i}\subset J_{u}$. From Proposition \ref{lem} and (\ref{07}) it
follows that 
\begin{equation*}
\left\vert J_{u}\right\vert -\left\vert G_{s}^{i}\right\vert
=2d_{k_{n}}-\left( d_{k_{n}-1}-3d_{k_{n}}\right) =\left( \delta _{n}-\left(
d_{k_{n}-1}-d_{k_{n}}\right) \right) +\left( 4d_{k_{n}}-\delta _{n}\right)
=2\delta _{n+1}.
\end{equation*}%
Thus, the condition $G_{s}^{i}\subset _{\ast }J_{u}$ is equivalent to the
condition%
\begin{equation*}
l\left( G_{s}^{i}\right) -l\left( J_{u}\right) =\delta _{n+1}\text{\quad
or\quad }r\left( J_{u}\right) -r\left( G_{s}^{i}\right) =\delta _{n+1}.
\end{equation*}%
We first prove (\ref{Th1-1}). The basic idea of the proof is to show (under
additional assumptions) that any gap $G_{s}^{i}$, where $s\in \left\{
0,1,2\right\} ^{k_{n}-1}$, $t\prec s$ and $i\in \left\{ 0,1\right\} $, is
covered by an interval $J_{u}$, where $u\in \left\{ 0,1,2\right\} ^{k_{n}}$
and $t\prec u$, that is, the following condition holds 
\begin{equation}
\exists _{u\in \left\{ 0,1,2\right\} ^{k_{n}}}\left( t\prec u\wedge
G_{s}^{i}\subset _{\ast }J_{u}\right) .  \tag{$\alpha _{n}$}
\end{equation}%
The proof is divided into a few steps. Let $m \in \N$ be such that $k_{m-1} \leq k < k_m$. Fix $n\geq m$.

\textbf{Claim 1. }\emph{If }$s\in \left\{ 0,1,2\right\} ^{k_{n}-1}$\emph{, }$%
t\prec s$\emph{, and }$i\in \left\{ 0,1\right\} $\emph{, then}%
\begin{equation*}
\left( \exists _{l\in \left\{ m-1,\ldots ,n-1\right\} }\ \max \left(
k,k_{l}\right) <N\left( s,i\right) <k_{l+1}\right) \Rightarrow \left( \alpha
_{n}\right) .
\end{equation*}

Let us write $N:=N\left( s,i\right) $ and consider two cases.

1$^{\text{o}}$ $i=0$. Then $N=\max \left\{ j:s_{j}>0\right\} $, $s_{N}\in
\left\{ 1,2\right\} $, and $s=p\symbol{94}s_{N}\symbol{94}0^{\left(
k_{n}-1-N\right) }$, where $p\in \left\{ 0,1,2\right\} ^{N-1}$. Define a
sequence $u\in \left\{ 0,1,2\right\} ^{k_{n}}$ in the following way: 
\begin{equation*}
u_{j}:=\left\{ 
\begin{array}{lll}
s_{j} & \text{if} & j\leq N-1 \\ 
s_{N}-1 & \text{if} & j=N \\ 
1 & \text{if} & j\in \left\{ k_{l+1},\ldots ,k_{n-1}\right\} \\ 
2 & \text{if} & j>N,j\notin \left\{ k_{l+1},\ldots ,k_{n-1}\right\}%
\end{array}%
\right. .
\end{equation*}%
Since $k<N$, we have $t\prec s|\left( N-1\right) \prec u$. Moreover, $%
G_{s}^{0}\subset _{\ast }J_{u}$ because from (\ref{07}) and Proposition \ref%
{lem} it follows that%
\begin{eqnarray*}
l\left( G_{s}^{0}\right) -l\left( J_{u}\right) &=&l\left( J_{s\symbol{94}%
0}\right) +2d_{k_{n}}-l\left( J_{u}\right) \\
&=&2d_{k_{n}}+\left( d_{N-1}-d_{N}\right) -\sum_{r=N+1}^{k_{n}}2\left(
d_{r-1}-d_{r}\right) +\sum_{r=l+1}^{n-1}\left(
d_{k_{r}-1}-d_{k_{r}}\right) \\
&=&2d_{k_{n}}+d_{N-1}-d_{N}-2d_{N}+2d_{k_{n}}+\sum_{r=l+1}^{n-1}\left(
\delta _{r}-\delta _{r+1}\right) \\
&=&4d_{k_{n}}-\delta _{l+1}+\sum_{r=l+1}^{n-1}\left( \delta _{r}-\delta
_{r+1}\right) =4d_{k_{n}}-\delta _{n}=\delta _{n+1}.
\end{eqnarray*}

2$^{\text{o}}$ $i=1$. Then $N=\max \left\{ j:s_{j}<2\right\} $, $s_{N}\in
\left\{ 0,1\right\} $, and $s=p\symbol{94}s_{N}\symbol{94}2^{\left(
k_{n}-1-N\right) }$, where $p\in \left\{ 0,1,2\right\} ^{N-1}$. Let us
define a sequence $u\in \left\{ 0,1,2\right\} ^{k_{n}}$ in the following
way: 
\begin{equation*}
u_{j}:=\left\{ 
\begin{array}{lll}
s_{j} & \text{if} & j\leq N-1 \\ 
s_{N}+1 & \text{if} & j=N \\ 
1 & \text{if} & j\in \left\{ k_{l+1},\ldots ,k_{n-1}\right\} \\ 
0 & \text{if} & j>N,j\notin \left\{ k_{l+1},\ldots ,k_{n-1}\right\}%
\end{array}%
\right. .
\end{equation*}%
Then $t\prec u$ and $G_{s}^{1}\subset _{\ast }J_{u}$ because%
\begin{eqnarray*}
r\left( J_{u}\right) -r\left( G_{s}^{1}\right) &=&r\left( J_{u}\right)
-r\left( J_{s\symbol{94}2}\right) +2d_{k_{n}} \\
&=&2d_{k_{n}}+\left( d_{N-1}-d_{N}\right) -\sum_{r=N+1}^{k_{n}}2\left(
d_{r-1}-d_{r}\right) +\sum_{r=l+1}^{n-1}\left(
d_{k_{r}-1}-d_{k_{r}}\right) \\
&=&2d_{k_{n}}+d_{N-1}-d_{N}-2d_{N}+2d_{k_{n}}+\sum_{r=l+1}^{n-1}\left(
\delta _{r}-\delta _{r+1}\right) \\
&=&4d_{k_{n}}-\delta _{l+1}+\sum_{r=l+1}^{n-1}\left( \delta _{r}-\delta
_{r+1}\right) =4d_{k_{n}}-\delta _{n}=\delta _{n+1}.
\end{eqnarray*}%
This finishes the proof of Claim 1.

\textbf{Claim 2.} \emph{If }$s\in \left\{ 0,1,2\right\} ^{k_{n}-1}$\emph{, }$%
t\prec s$\emph{, }$i\in \left\{ 0,1\right\} $\emph{, and }$N:=N\left(
s,i\right) $\emph{, then}%
\begin{equation*}
\left( \exists _{l\in \left\{ m,\ldots ,n-1\right\} }\ \exists _{v\in
\left\{ 0,1,2\right\} ^{k_{l}}}\ N=k_{l}\wedge t\prec v\wedge G_{s|\left(
N-1\right) }^{s_{N}-1+i}\subset _{\ast }J_{v}\right) \Rightarrow \left(
\alpha _{n}\right) .
\end{equation*}

Let us consider two cases.

1$^{\text{o}}$ $i=0$. Then $N=\max \left\{ j:s_{j}>0\right\} $, $s_{N}\in
\left\{ 1,2\right\} $, and $s=p\symbol{94}s_{N}\symbol{94}0^{\left(
k_{n}-1-N\right) }$, where $p=s|\left( N-1\right) \in \left\{ 0,1,2\right\}
^{N-1}$. Observe that $l\left( J_{s}\right) =l\left( J_{s|N}\right) =r\left(
G_{p}^{s_{N}-1}\right) $. By assumption, $G_{p}^{s_{N}-1}\subset _{\ast
}J_{v}$. Define the sequence $u\in \left\{ 0,1,2\right\} ^{k_{n}}$ in the
following way: 
\begin{equation*}
u_{j}:=\left\{ 
\begin{array}{lll}
v_{j} & \text{if} & j\leq k_{l} \\ 
1 & \text{if} & j\in \left\{ k_{l+1},\ldots ,k_{n-1}\right\} \\ 
2 & \text{if} & j>k_{l},j\notin \left\{ k_{l+1},\ldots ,k_{n-1}\right\}%
\end{array}%
\right. .
\end{equation*}%
Since $k<k_{m}\leq k_{l}=N$, we have that $t\prec v|\left( N-1\right)
\prec u$. Moreover, $G_{s}^{0}\subset _{\ast }J_{u}$ because from (\ref{07}%
) and Proposition \ref{lem} it follows that 
\begin{eqnarray*}
l\left( G_{s}^{0}\right) -l\left( J_{u}\right) &=&\left( l\left(
J_{s}\right) +2d_{k_{n}}\right) -\left( r\left( J_{u}\right)
-2d_{k_{n}}\right) \\
&=&4d_{k_{n}}+\left( r\left( G_{p}^{s_{N}-1}\right) -r\left( J_{v}\right)
\right) +\left( r\left( J_{v}\right) -r\left( J_{u}\right) \right) \\
&=&4d_{k_{n}}-\delta _{l+1}+\sum_{r=l+1}^{n-1}\left( 2-1\right) \left(
d_{k_{r}-1}-d_{k_{r}}\right) \\
&=&4d_{k_{n}}-\delta _{l+1}+\sum_{r=l+1}^{n-1}\left( \delta _{r}-\delta
_{r+1}\right) =4d_{k_{n}}-\delta _{n}=\delta _{n+1}.
\end{eqnarray*}

2$^{\text{o}}$ $i=1$. Then $N=\max \left\{ j:s_{j}<2\right\} $, $s_{N}\in
\left\{ 0,1\right\} $, and $s=p\symbol{94}s_{N}\symbol{94}2^{\left(
k_{n}-1-N\right) }$, where $p=s|\left( N-1\right) \in \left\{ 0,1,2\right\}
^{N-1}$. Observe that $r\left( J_{s}\right) =r\left( J_{s|N}\right) =l\left(
G_{p}^{s_{N}}\right) $. By assumption, $G_{p}^{s_{N}}\subset _{\ast }J_{v}$.
Let us define the sequence $u\in \left\{ 0,1,2\right\} ^{k_{n}}$ in the
following way: 
\begin{equation*}
u_{j}:=\left\{ 
\begin{array}{lll}
v_{j} & \text{if} & j\leq k_{l} \\ 
1 & \text{if} & j\in \left\{ k_{l+1},\ldots ,k_{n-1}\right\} \\ 
0 & \text{if} & j>k_{l},j\notin \left\{ k_{l+1},\ldots ,k_{n-1}\right\}%
\end{array}%
\right. .
\end{equation*}%
Similarly as in case 1$^{\text{o}}$, we get $t\prec u$ and $r\left(
J_{u}\right) -r\left( G_{s}^{1}\right) =\delta _{n+1}$, so $G_{s}^{1}\subset
_{\ast }J_{u}$. This completes the proof of Claim 2.

\textbf{Claim 3. }\emph{If }$n\geq m$\emph{, }$s\in \left\{ 0,1,2\right\}
^{k_{n}-1}$\emph{, }$t\prec s$\emph{, and }$i\in \left\{ 0,1\right\} $\emph{%
, then }%
\begin{equation}
G_{s}^{i}\in \mathcal{G}_{t}\left( n\right) \text{\quad or\quad }\emph{(}%
\alpha _{n}\emph{)}.  \label{08c}
\end{equation}

We prove Claim 3 inductively. Let $n=m$. If $k=k_{m}-1$, then $s=t$ and $%
\mathcal{G}_{t}\left( m\right) =\left\{ G_{t}^{0},G_{t}^{1}\right\} $, so (%
\ref{08c}) holds. If $k<k_{m}-1$ and $G_{s}^{i}\notin \mathcal{G}_{t}\left(
m\right) =\left\{ G_{t\symbol{94}0^{\left( k_{m}-k-1 \right) }}^{0},G_{t\symbol{94}%
2^{\left( k_{m}-k-1 \right) }}^{1}\right\} $, then $N\left( s,i\right) >k$ and from Claim 1
we infer that ($\alpha _{m}$) holds. Hence in this case, condition (\ref{08c}%
) is also satisfied.

Assume now that $n\geq m$ and for any $l\in \left\{ m,\ldots ,n\right\} $,
any $s\in \left\{ 0,1,2\right\} ^{k_{l}-1}$ such that $t\prec s$, and any $%
i\in \left\{ 0,1\right\} $, $G_{s}^{i}\in \mathcal{G}_{t}\left( l\right) $\
or ($\alpha _{l}$) holds. Let $s\in \left\{ 0,1,2\right\} ^{k_{n+1}-1}$, $%
t\prec s$, $i\in \left\{ 0,1\right\} $, and $N:=N\left( s,i\right) $. We can
assume that $i=0$ (for $i=1$, the proof is analogous). If $N\leq k$, then $s=t%
\symbol{94}0^{\left( k_{n+1}-k-1\right) }$, and consequently $G_{s}^{0}\in 
\mathcal{G}_{t}\left( n+1\right) $. Let us assume that $N>k$ and consider
two cases.

1$^{\text{o}}$ There exists $l\in \left\{ m-1,\ldots ,n\right\} $ such that $%
k_{l}<N<k_{l+1}$. Then from Claim 1 it follows that ($\alpha _{n+1}$) is
satisfied, so condition (\ref{08c}) holds for $n+1$.

2$^{\text{o}}$ There exists $l\in \left\{ m,\ldots ,n\right\} $ such that $%
N=k_{l}$. If there is a sequence $v\in \left\{ 0,1,2\right\} ^{k_{l}}$ such
that $t\prec v\ $and $G_{s|\left( N-1\right) }^{s_{N}-1}\subset _{\ast }J_{v}
$, then Claim 2 implies ($\alpha _{n+1}$). In the other case, for the
sequence $s|\left( N-1\right) $ and $i=s_{N}-1$, condition ($\alpha _{l}$)
is not satisfied. Thus, the inductive hypothesis leads to $G_{s|\left(
N-1\right) }^{s_{N}-1}\in \mathcal{G}_{t}\left( l\right) $. If $s_{N}=1$,
 then $s=\left( s|\left( N-1\right) \right) \symbol{94}1\symbol{94}0^{\left(
k_{n+1}-N-1\right) }$ and $G_{s|\left( N-1\right) }^{0}\in \mathcal{G}%
_{t}\left( l\right) $, while if $s_{N}=2$, then $s=\left( s|\left( N-1\right)
\right) \symbol{94}2\symbol{94}0^{\left( k_{n+1}-N-1\right) }$ and $%
G_{s|\left( N-1\right) }^{1}\in \mathcal{G}_{t}\left( l\right) $. In both
cases, $G_{s}^{0}\in \mathcal{G}_{t}\left( n+1\right) $. This finishes the
proof of (\ref{08c}) for $n+1$.

\textbf{Claim 4. }\emph{Let }$x\in J_{t}\setminus \bigcup G_{t}$\emph{. For
any }$n\geq m$\emph{, we have }%
\begin{equation}
\exists _{u\in \left\{ 0,1,2\right\} ^{k_{n}}}\left( t\prec u\wedge x\in
J_{u}\right) .  \label{09}
\end{equation}%
We prove Claim 4 inductively. Let $n=m$. Since $\lambda _{k+1},\ldots
,\lambda _{k_{m}-1}\geq \frac{1}{3}$, Proposition \ref{lem}(\ref{1-6})
implies that there exists a sequence $s\in \left\{ 0,1,2\right\} ^{k_{m}-1}$
such that $t\prec s\ $ and $x\in J_{s}$ (if $k=k_{m}-1$, we put $s=t$).
Hence $x\in J_{s\symbol{94}0}\cup J_{s\symbol{94}1}\cup J_{s\symbol{94}2}$
or $x\in G_{s}^{0}\cup G_{s}^{1}$. In the first case, condition (\ref{09})
is satisfied for $u=s\symbol{94}j$. Let us assume that $x\in G_{s}^{0}\cup
G_{s}^{1}$. We may assume that $x\in G_{s}^{0}$ (if $x\in G_{s}^{1}$, the
proof is analogous). Since $x\notin \bigcup \mathcal{G}_{t}$, $%
G_{s}^{0}\notin \mathcal{G}_{t}\left( m\right) $. Using Claim 3 we deduce
that there is a sequence $u\in \left\{ 0,1,2\right\} ^{k_{n}}$ such that $%
t\prec u\ $and $G_{s}^{i}\subset _{\ast }J_{u}$. This completes the proof of
(\ref{09}) for $n=m$.

Assume now that $n\geq m$ and condition (\ref{09}) holds for $n$, i.e. there
is a sequence $v\in \left\{ 0,1,2\right\} ^{k_{n}}$ such that $t\prec v\ $
and $x\in J_{v}$. Since $\lambda _{k_{n}+1},\ldots ,\lambda _{k_{n+1}-1}\geq 
\frac{1}{3}$, Proposition \ref{lem}(\ref{1-6}) shows that there is a
sequence $s\in \left\{ 0,1,2\right\} ^{k_{n+1}-1}$ such that $t\prec s\ $
and $x\in J_{s}$. From Claim 3 and ($\alpha _{n+1}$) we infer the existence
of a sequence $u\in \left\{ 0,1,2\right\} ^{k_{n+1}}$ such that $t\prec u\ $
and $x\in J_{u}$, which finishes the proof of (\ref{09})\ for $n+1$.

From Claim 4 it follows that $J_{t}\setminus \bigcup \mathcal{G}_{t}\subset
\bigcup_{u\in \left\{ 0,1,2\right\} ^{k_{n}},t\prec u}J_{u}\subset
C_{k_{n}}\left( \lambda \right) -C_{k_{n}}\left( \lambda \right) $ for $%
n\geq m$, and consequently 
\begin{equation*}
J_{t}\setminus \bigcup \mathcal{G}_{t}\subset \bigcap_{n=m}^{\infty }\left(
C_{k_{n}}\left( \lambda \right) -C_{k_{n}}\left( \lambda \right) \right)
=C\left( \lambda \right) -C\left( \lambda \right) ,
\end{equation*}%
so (\ref{Th1-1}) holds. From Theorem \ref{tw4} we deduce that the set $%
C\left( \lambda \right) -C\left( \lambda \right) $ has nonempty interior
and by Theorems \ref{AI} and \ref{tw1} it is a Cantorval. If $t=\emptyset $,
then Theorem \ref{tw4} implies $C\left( \lambda \right) -C\left( \lambda
\right) \subset \left[ -1,1\right] \setminus \bigcup \mathcal{G}_{t}$.\ Thus, 
$C\left( \lambda \right) -C\left( \lambda \right) =\left[ -1,1\right]
\setminus \bigcup \mathcal{G}_{t}$. Using once again Theorem \ref{tw4}, we get 
\begin{equation*}
\left\vert C\left( \lambda \right) -C\left( \lambda \right) \right\vert
=2-\left\vert \bigcup \mathcal{G}_{t}\right\vert =2-2\sum_{n=1}^{\infty
}3^{n-1}\left( d_{k_{n}-1}-3d_{k_{n}}\right) .
\end{equation*}
\end{proof}
\end{theorem}

\begin{remark}
\label{r1}If a sequence $\left( \delta _{n}\right) $ satisfies condition (%
\ref{07}), then 
\begin{equation*}
\delta _{n}=3d_{k_{0}+1}-d_{k_{0}}-\sum_{r=1}^{n-1}\left(
d_{k_{r}-1}-d_{k_{r}}\right) =\frac{1}{2}\left(
3d_{k_{n}}+d_{k_{n}-1}\right) .
\end{equation*}
\end{remark}

The next proposition shows that the assumptions of Theorem \ref{Th1} can be
written as a system of infinitely many equations. In the next section we
will use such a version of condition (\ref{07}) to examine achievement sets
of some fast convergent series.

\begin{proposition}
\label{Pr1}Assume that $\lambda =\left( \lambda _{j}\right) _{j\in \mathbb{N}%
}\in \left( 0,\frac{1}{2}\right) ^{\mathbb{N}}$ is a sequence such that $%
\lambda _{n}<\frac{1}{3}$\ for infinitely many terms and there is $k_{0}\in 
\mathbb{N}\cup \left\{ 0\right\} $ such that $\lambda _{k_{0}+1}>\frac{1}{3}$%
. Then the following conditions are equivalent.

\begin{enumerate}
\item There is a sequence $\left( \delta _{n}\right) _{n\in \mathbb{N}}$
satisfying condition (\ref{07}).

\item For any $r\in \mathbb{N}$, we have: 
\begin{equation*}
\begin{array}{lll}
3d_{r+1}=4d_{r}-d_{r-1} & \text{if} & \lambda _{r},\lambda _{r+1}\geq \frac{1%
}{3}\text{ or }\lambda _{r},\lambda _{r+1}<\frac{1}{3}, \\ 
3d_{r+1}=5d_{r}-2d_{r-1} & \text{if} & \lambda _{r}\geq \frac{1}{3}\ \text{%
and }\lambda _{r+1}<\frac{1}{3}, \\ 
6d_{r+1}=7d_{r}-d_{r-1} & \text{if} & \lambda _{r}<\frac{1}{3}\text{ and }%
\lambda _{r+1}\geq \frac{1}{3}.%
\end{array}%
\end{equation*}

\item \label{Pr1-3}For any $r\in \mathbb{N}$, we have:%
\begin{equation*}
\begin{array}{lll}
3\lambda _{r}\lambda _{r+1}=4\lambda _{r}-1 & \text{if} & \lambda
_{r},\lambda _{r+1}\geq \frac{1}{3}\text{ or }\lambda _{r},\lambda _{r+1}<%
\frac{1}{3}, \\ 
3\lambda _{r}\lambda _{r+1}=5\lambda _{r}-2 & \text{if} & \lambda _{r}\geq 
\frac{1}{3}\ \text{and }\lambda _{r+1}<\frac{1}{3}, \\ 
6\lambda _{r}\lambda _{r+1}=7\lambda _{r}-1 & \text{if} & \lambda _{r}<\frac{%
1}{3}\text{ and }\lambda _{r+1}\geq \frac{1}{3}.%
\end{array}%
\end{equation*}
\end{enumerate}

\begin{proof}
The equivalence of conditions (2) and (3) is obvious.

(1)$\Rightarrow $(2). Let us consider four cases.

1$^{\text{o}}$ $\lambda _{r},\lambda _{r+1}\geq \frac{1}{3}$. Then $%
3d_{r}-d_{r-1}=3d_{r+1}-d_{r}$, so $3d_{r+1}=4d_{r}-d_{r-1}$.

2$^{\text{o}}$ $\lambda _{r},\lambda _{r+1}<\frac{1}{3}$.\ Then $r=k_{n}$
and\ $r+1=k_{n+1}$ for some $n\in \mathbb{N}$. Hence%
\begin{equation*}
\left\{ 
\begin{array}{l}
-4d_{r}=-4d_{k_{n}}=-\delta _{n}-\delta _{n+1}, \\ 
d_{r-1}-d_{r}=d_{k_{n}-1}-d_{k_{n}}=\delta _{n}-\delta _{n+1}, \\ 
4d_{r+1}=4d_{k_{n+1}}=\delta _{n+1}+\delta _{n+2}, \\ 
d_{r}-d_{r+1}=d_{k_{n+1}-1}-d_{k_{n+1}}=\delta _{n+1}-\delta _{n+2}.%
\end{array}%
\right.
\end{equation*}%
Adding the above equations, we get $3d_{r+1}-4d_{r}+d_{r-1}=0$.

3$^{\text{o }}\lambda _{r}\geq \frac{1}{3}\ $and $\lambda _{r+1}<\frac{1}{3}$%
. Then $r=k_{n}-1$ for some $n\in \mathbb{N}$. Hence 
\begin{equation*}
\left\{ 
\begin{array}{l}
4d_{r+1}=4d_{k_{n}}=\delta _{n}+\delta _{n+1}, \\ 
d_{r}-d_{r+1}=d_{k_{n}-1}-d_{k_{n}}=\delta _{n}-\delta _{n+1}, \\ 
-2\left( 3d_{r}-d_{r-1}\right) =-2\delta _{n}.%
\end{array}%
\right.
\end{equation*}%
Adding the above equations, we obtain $3d_{r+1}-5d_{r}+2d_{r-1}=0$.

4$^{\text{o}}$ $\lambda _{r}<\frac{1}{3}$ and $\lambda _{r+1}\geq \frac{1}{3}
$. Then $r=k_{n}$ for some $n\in \mathbb{N}$. Hence 
\begin{equation*}
\left\{ 
\begin{array}{l}
-4d_{r}=-4d_{k_{n}}=-\delta _{n}-\delta _{n+1}, \\ 
d_{r-1}-d_{r}=d_{k_{n}-1}-d_{k_{n}}=\delta _{n}-\delta _{n+1}, \\ 
2\left( 3d_{r+1}-d_{r}\right) =2\delta _{n+1}.%
\end{array}%
\right.
\end{equation*}%
Adding the above equations, we get $6d_{r+1}-7d_{r}+d_{r-1}=0$.

(2)$\Rightarrow $(1).\textbf{\ }Set $\delta _{n}:=\frac{1}{2}\left(
3d_{k_{n}}+d_{k_{n}-1}\right) $. We will show that this sequence satisfies
condition (\ref{07}).

Assume that $k_{n-1}<r<k_{n}$. If $r+1<k_{n}$, then $\lambda _{r},\lambda
_{r+1}\geq \frac{1}{3}$. Therefore, $3d_{r+1}=4d_{r}-d_{r-1}$, and
consequently $3d_{r+1}-d_{r}=3d_{r}-d_{r-1}$. Hence 
\begin{equation*}
3d_{k_{n-1}+1}-d_{k_{n-1}}=\ldots =3d_{k_{n}-1}-d_{k_{n}-2}.
\end{equation*}%
On the other hand, if $r+1=k_{n}$, then $\lambda _{k_{n}-1}\geq \frac{1}{3}$
and $\lambda _{k_{n}}<\frac{1}{3}.$ Therefore, $%
3d_{k_{n}}=5d_{k_{n}-1}-2d_{k_{n}-2}$, and consequently 
\begin{equation*}
3d_{k_{n}-1}-d_{k_{n}-2}=\frac{1}{2}\left( 6d_{k_{n}-1}-2d_{k_{n}-2}\right) =%
\frac{1}{2}\left( 3d_{k_{n}}+d_{k_{n}-1}\right) =\delta _{n},
\end{equation*}%
which completes the proof of the first equality in (\ref{07}).

We now show that the second equality holds. If $k_{n+1}-k_{n}=1$, then $%
\lambda _{k_{n}},\lambda _{k_{n}+1}<\frac{1}{3}$, and consequently $%
3d_{k_{n}+1}=4d_{k_{n}}-d_{k_{n}-1}$. Hence%
\begin{eqnarray*}
4d_{k_{n}}-\delta _{n} &=&4d_{k_{n}}-\frac{1}{2}\left(
3d_{k_{n}}+d_{k_{n}-1}\right) =4d_{k_{n}}-\frac{1}{2}\left(
3d_{k_{n}}+4d_{k_{n}}-3d_{k_{n}+1}\right) \\
&=&\frac{1}{2}\left( 3d_{k_{n}+1}+d_{k_{n}}\right) =\frac{1}{2}\left(
3d_{k_{n+1}}-d_{k_{n+1}-1}\right) =\delta _{n+1}.
\end{eqnarray*}%
On the other hand, if $k_{n+1}-k_{n}>1$, then $\lambda _{k_{n}}<\frac{1}{3}$
and $\lambda _{k_{n}+1}\geq \frac{1}{3}$, and consequently $%
6d_{k_{n}+1}=7d_{k_{n}}-d_{k_{n}-1}$. Using the first equality, we get 
\begin{eqnarray*}
4d_{k_{n}}-\delta _{n} &=&4d_{k_{n}}-\frac{1}{2}\left(
3d_{k_{n}}+d_{k_{n}-1}\right) =4d_{k_{n}}-\frac{1}{2}\left(
3d_{k_{n}}+7d_{k_{n}}-6d_{k_{n}+1}\right) \\
&=&3d_{k_{n}+1}-d_{k_{n}}=\delta _{n+1},
\end{eqnarray*}%
which finishes the proof of the second equality in (\ref{07}). The third
equality results from the second one because 
\begin{equation*}
\delta _{n}-\delta _{n+1}=2\delta _{n}-\left( \delta _{n}+\delta
_{n+1}\right) =3d_{k_{n}}+d_{k_{n}-1}-4d_{k_{n}}=d_{k_{n}-1}-d_{k_{n}}.
\end{equation*}
\end{proof}
\end{proposition}
It may seem unnatural to describe covering the gaps via equalities as in Theorem \ref{Th1}. Our initial idea was to describe it by inequalities. However, it turned out that the system of received inequalities is equivalent to condition (\ref{07}). It suggests that sequences satisfying the assumptions of Theorem \ref{Th1} are quite exceptional. In the next section we will argue why the family of such sequences is interesting.
\section{The sets os subsums of series}

We will use the results obtained in the previous section to examine the sets
of subsums of series. Let $x=\left( x_{j}\right) _{j\in \mathbb{N}}$ be
a nonincreasing sequence of positive numbers such that the series $%
\sum_{j=1}^{\infty }x_{j}$ is convergent. The set%
\begin{equation*}
E\left( x\right) :=\left\{ \sum_{j\in A}x_{j}:A\subset \mathbb{N}\right\} 
\end{equation*}%
of all subsums of $\sum_{j=1}^{\infty }x_{j}$ is called \emph{the
achievement set} of $x$. The sum and the remainders of a series we denote by 
$S:=\sum_{j=1}^{\infty }x_{j}$ and $r_{n}:=\sum_{j=n+1}^{\infty }x_{j}$.
Obviously, $r_{0}=S$. If $x_{n}>r_{n}$ for $n\in \mathbb{N}$, then the series
is called \emph{fast convergent}.

There is a very close relationship between central Cantor sets and the
achievement sets of fast convergent series. This is shown in the following proposition.

\begin{proposition}[{\protect\cite[p. 27]{FF}}]
\label{Pr2}The following conditions hold.

\begin{enumerate}
\item If $\left( \lambda _{j}\right) _{j\in \mathbb{N}}\in \left( 0,\frac{1}{%
2}\right) ^{\mathbb{N}}$, then the series $\sum_{j=1}^{\infty }x_{j}$ given
by the formula%
\begin{equation}
x_{1}=1-\lambda _{1}\quad \text{and}\quad x_{j}=\lambda _{1}\cdot \ldots
\cdot \lambda _{j-1}\cdot \left( 1-\lambda _{j}\right) \text{ for }j>1,
\label{3-1}
\end{equation}%
is fast convergent, $S=1$ and $C\left( \lambda \right) =E\left( x\right) $.

\item If a series $\sum_{j=1}^{\infty }x_{j}$ is fast convergent and $%
\lambda _{j}=\frac{r_{j}}{r_{j-1}}$ for $j\in \mathbb{N}$, then $\left(
\lambda _{j}\right) _{j\in \mathbb{N}}\in \left( 0,\frac{1}{2}\right) ^{%
\mathbb{N}}$ and $E\left( x\right) =S\cdot C\left( \lambda \right) $.
\end{enumerate}
\end{proposition}

We will examine the difference set $E\left( x\right) -E\left( x\right) $ for
some convergent series with positive terms. Observe that $\left( E\left(
x\right) -E\left( x\right) \right) +S=E\left( x\right) +E\left( x\right)
=E\left( x_{1},x_{1},x_{2},x_{2},\ldots \right) $. Thus, the properties of
the difference $E\left( x\right) -E\left( x\right) $ are the same as the
properties of the sum $E\left( x\right) +E\left( x\right) $.

We first check which series correspond to sequences satisfying the
assumptions of Theorem \ref{Th1}. It turns out that they have a very simple
form.

\begin{proposition}
Assume that a sequence $\lambda \in \left( 0,\frac{1}{2}\right) ^{\mathbb{N}%
} $ satisfies the assumptions of Theorem \ref{Th1}, that is $\lambda _{n}<%
\frac{1}{3}$\ for infinitely many terms, $\lambda _{n}\geq \frac{1}{3}$ for
infinitely many terms, $\lambda _{k_{0}+1}>\frac{1}{3}$ for some $k_{0}\in 
\mathbb{N}\cup \left\{ 0\right\} $, $\left( k_{n}\right) _{n\in \mathbb{N}}$
is an increasing sequence of natural numbers such that $\left( \lambda
_{k_{n}}\right) _{n\in \mathbb{N}}$ is a subsequence of the sequence $\left(
\lambda _{j}\right) _{j>k_{0}}$ consisting of all terms which are less than $%
\frac{1}{3}$, and there is a sequence $\left( \delta _{n}\right) $
satisfying (\ref{07}). If the sequence $x$ is given by formula (\ref{3-1}), 
then for any $j\in \mathbb{N}$, we have%
\begin{equation*}
x_{k_{0}+j}=\left\{ 
\begin{array}{c}
\frac{1}{3^{j-1}}x_{k_{0}+1}\text{ if }\lambda _{k_{0}+j}\geq \frac{1}{3} \\ 
\frac{2}{3^{j-1}}x_{k_{0}+1}\text{ if }\lambda _{k_{0}+j}<\frac{1}{3}%
\end{array}%
\right. .
\end{equation*}

\begin{proof}
Using Proposition \ref{Pr1} we will calculate $\frac{x_{r+1}}{x_{r}}$ for $%
r>k_{0}$. We have 
\begin{equation*}
\frac{x_{r+1}}{x_{r}}=\frac{\lambda _{r}\left( 1-\lambda _{r+1}\right) }{%
1-\lambda _{r}}.
\end{equation*}%
If $\lambda _{r},\lambda _{r+1}\geq \frac{1}{3}\ $or $\lambda _{r},\lambda
_{r+1}<\frac{1}{3}$, then $3\lambda _{r}\lambda _{r+1}=4\lambda _{r}-1$.
Hence $1-\lambda _{r+1}=\frac{1-\lambda _{r}}{3\lambda _{r}}$, and so $\frac{%
x_{r+1}}{x_{r}}=\frac{1}{3}$. If $\lambda _{r}\geq \frac{1}{3}$ and $\lambda
_{r+1}<\frac{1}{3}$, then $3\lambda _{r}\lambda _{r+1}=5\lambda _{r}-2$.
Therefore, $1-\lambda _{r+1}=\frac{2-2\lambda _{r}}{3\lambda _{r}}$, which
gives $\frac{x_{r+1}}{x_{r}}=\frac{2}{3}$. If $\lambda _{r}<\frac{1}{3}$ and 
$\lambda _{r+1}\geq \frac{1}{3}$, then $6\lambda _{r}\lambda _{r+1}=7\lambda
_{r}-1$. Thus, $1-\lambda _{r+1}=\frac{1-\lambda _{r}}{6\lambda _{r}}$, and
consequently $\frac{x_{r+1}}{x_{r}}=\frac{1}{6}$. An easy induction leads to
our assertion.
\end{proof}
\end{proposition}

\begin{corollary}\label{col1}
Assume that a sequence $\lambda $ satisfies the assumptions of Theorem \ref%
{Th1} and $k_{0}=0$, i.e. $\lambda _{1}>\frac{1}{3}$ and the sequence $%
\left( \lambda _{k_{n}}\right) _{n\in \mathbb{N}}$ consists of all terms of
the sequence $\lambda $ less than $\frac{1}{3}$. If the sequence $x$ is
given by formula (\ref{3-1}), then $x_{j}=\left\{ 
\begin{array}{c}
\frac{2x_{1}}{3^{j-1}}\text{ if }j\in \left\{ k_{n}:n\in \mathbb{N}\right\}
\\ 
\frac{x_{1}}{3^{j-1}}\text{ if }j\notin \left\{ k_{n}:n\in \mathbb{N}\right\}%
\end{array}%
\text{.}\right. $
\end{corollary}

We will now show that the implication from Corollary \ref{col1} can be
conversed. Namely, every increasing sequence $\left( k_{n}\right) _{n\in 
\mathbb{N}}$ of natural numbers, for which $k_{1}>1$ and the set $\mathbb{N}%
\setminus \left\{ k_{n}:n\in \mathbb{N}\right\} $ is infinite, generates a
fast convergent series $\sum_{j=1}^{\infty }x_{j}$\ and the corresponding
sequence $\lambda $ satisfying the assumptions of Theorem \ref{Th1} (for $%
k_{0}=0$).

\begin{theorem}
\label{th5}Let $\left( k_{n}\right) _{n\in \mathbb{N}}$ be an increasing
sequence of natural numbers such that $k_{1}>1$ and the set $\mathbb{N}%
\setminus \left\{ k_{n}:n\in \mathbb{N}\right\} $ is infinite. Put 
\begin{equation*}
x_{j}:=\left\{ 
\begin{array}{ccc}
\frac{2}{3^{j-1}} & \text{if} & j\in \left\{ k_{n}:n\in \mathbb{N}\right\}
\\ 
\frac{1}{3^{j-1}} & \text{if} & j\notin \left\{ k_{n}:n\in \mathbb{N}\right\}%
\end{array}%
\right. ,
\end{equation*}%
$S:=\sum_{j=1}^{\infty }x_{j}$, $r_{n}:=\sum_{j=n+1}^{\infty }x_{j}$, and $%
\lambda _{n}:=\frac{r_{n}}{r_{n-1}}$ for $n\in \mathbb{N}$. The following
conditions hold.

\begin{enumerate}
\item The sequence $\lambda $ is the only sequence satisfying the
assumptions of Theorem \ref{Th1} for $k_{0}=0$. It means that $\left( \lambda
_{j} \right)\in \left( 0,\frac{1}{2}\right)^{\N} $, $\lambda _{1}>\frac{1}{3}$, $\lambda
_{k_{n}}<\frac{1}{3}$, $\lambda _{j}\geq \frac{1}{3}$ if $j\notin \left\{
k_{n}:n\in \mathbb{N}\right\} $, and there exists a sequence $\left( \delta
_{n}\right) $ satisfying condition (\ref{07}).

\item The set $E\left( x\right) -E\left( x\right) $ is a Cantorval and $%
E\left( x\right) =S\cdot C\left( \lambda \right) $.
\end{enumerate}

\begin{proof}
We have $\frac{3}{2}<S<3$ and 
\begin{equation*}
\frac{1}{2\cdot 3^{n-1}}=\sum_{j=n+1}^{\infty }\frac{1}{3^{j-1}}%
<r_{n}<\sum_{j=n+1}^{\infty }\frac{2}{3^{j-1}}=\frac{1}{3^{n-1}}\leq x_{n}.
\end{equation*}%
Thus, the series $\sum_{j=1}^{\infty }x_{j}$ is fast convergent and from
Proposition \ref{Pr2} it follows that $E\left( x\right) =S\cdot C\left(
\lambda \right) $ (in particular, $\left( \lambda _{j}\right) \in \left( 0,\frac{1}{2}%
\right)^{\N} $). If $j\in \left\{ k_{n}:n\in \mathbb{N}\right\} $, then $x_{j}=%
\frac{2}{3^{j-1}}$, and consequently%
\begin{equation*}
\lambda _{j}=\frac{r_{j}}{r_{j-1}}=\frac{r_{j-1}-x_{j}}{r_{j-1}}=1-\frac{%
x_{j}}{r_{j-1}}<1-\frac{x_{j}}{\frac{1}{3^{j-2}}}=\frac{1}{3}.
\end{equation*}%
However, if $j\notin \left\{ k_{n}:n\in \mathbb{N}\right\} $, then $x_{j}=%
\frac{1}{3^{j-1}}$, which yields%
\begin{equation*}
\lambda _{j}=1-\frac{x_{j}}{r_{j-1}}>1-\frac{x_{j}}{\frac{1}{2\cdot 3^{j-2}}}%
=\frac{1}{3}.
\end{equation*}%
Since $k_{1}>1$, we have $\lambda _{1}>\frac{1}{3}$. We will prove that condition (%
\ref{Pr1-3}) from Proposition \ref{Pr1} holds. If $\lambda _{j},\lambda
_{j+1}\geq \frac{1}{3}$ or $\lambda _{j},\lambda _{j+1}<\frac{1}{3}$, then $%
x_{j+1}=\frac{1}{3}x_{j}$, and, in consequence, 
\begin{equation*}
3\lambda _{j}\lambda _{j+1}-4\lambda _{j}+1=\frac{3r_{j+1}}{r_{j-1}}-\frac{%
4r_{j}}{r_{j-1}}+1=\frac{3r_{j+1}-4r_{j}+r_{j-1}}{r_{j-1}}=\frac{%
x_{j}-3x_{j+1}}{r_{j-1}}=0.
\end{equation*}%
If $\lambda _{j}\geq \frac{1}{3}$ and $\lambda _{j+1}<\frac{1}{3}$, then $%
x_{j+1}=\frac{2}{3}x_{j}$, which gives 
\begin{equation*}
3\lambda _{j}\lambda _{j+1}-5\lambda _{j}+2=\frac{3r_{j+1}-5r_{j}+2r_{j-1}}{%
r_{j-1}}=\frac{2x_{j}-3x_{j+1}}{r_{j-1}}=0.
\end{equation*}%
If $\lambda _{j}<\frac{1}{3}$ and $\lambda _{j+1}\geq \frac{1}{3}$, then $%
x_{j+1}=\frac{1}{6}x_{j}$, and consequently 
\begin{equation*}
6\lambda _{j}\lambda _{j+1}-7\lambda _{j}+1=\frac{6r_{j+1}-7r_{j}+r_{j-1}}{%
r_{j-1}}=\frac{x_{j}-6x_{j+1}}{r_{j-1}}=0.
\end{equation*}%
By Proposition \ref{Pr1} there is a sequence $\left( \delta _{n}\right) $
satisfying condition (\ref{07}). The explicitness of $\left( \delta
_{n}\right) $ follows from Remark \ref{r1}. Theorem \ref{Th1} implies that $%
C\left( \lambda \right) -C\left( \lambda \right) $ is a Cantorval, which
finishes the proof.
\end{proof}
\end{theorem}

If $x_{j}=\frac{1}{3^{j-1}}$ for all $j\in \mathbb{N}$, then $E\left(
x\right) $ is the Cantor ternary set (in the interval $\left[ 0,\frac{3}{2}%
\right] $). It is easy to check that, if $x_{j}=\frac{1}{3^{j-1}}$ for almost
all $j$'s and $x_{j}=\frac{2}{3^{j-1}}$ for the remaining ones, then $%
E\left( x\right) $ is a finite union of the Cantor ternary sets.
Consequently, $E\left( x\right) -E\left( x\right) $ is a finite union of
closed intervals. We get an analogous result if $x_{j}=\frac{2}{3^{j-1}}$
for almost all $j$'s and $x_{j}=\frac{1}{3^{j-1}}$ for the remaining ones.

Using Theorem \ref{Th1}, we can calculate the measure of Cantorvals $E\left(
x\right) -E\left( x\right) $ considered in Theorem \ref{th5}.

\begin{theorem}
\label{tw6}Let $\left( k_{n}\right) _{n\in \mathbb{N}}$ be an increasing
sequence of natural numbers such that $k_{1}>1$ and the set $\mathbb{N}%
\setminus \left\{ k_{n}:n\in \mathbb{N}\right\} $ is infinite. Put $%
x_{j}:=\left\{ 
\begin{array}{ccc}
\frac{2}{3^{j-1}} & \text{if} & j\in \left\{ k_{n}:n\in \mathbb{N}\right\}
\\ 
\frac{1}{3^{j-1}} & \text{if} & j\notin \left\{ k_{n}:n\in \mathbb{N}\right\}%
\end{array}%
\right. $. Then $\left\vert E\left( x\right) +E\left( x\right) \right\vert
=\left\vert E\left( x\right) -E\left( x\right) \right\vert =3$.

\begin{proof}
Let $S:=\sum_{n=1}^{\infty }x_{n}$ and $y_{n}:=\sum_{k_{n}<j<k_{n+1}}x_{j}$ (%
$y_{n}:=0$ if $k_{n+1}=k_{n}+1$). Then 
\begin{equation*}
y_{n}=\sum_{j=k_{n}+1}^{k_{n+1}-1}\frac{1}{3^{j}}=\frac{1}{2}\left( \frac{1}{%
3^{k_{n}}}-\frac{1}{3^{k_{n+1}-1}}\right)
\end{equation*}%
and for $n\in \mathbb{N}$ we have%
\begin{multline*}
2\left( y_{n}+3y_{n+1}+3^{2}y_{n+2}+3^{3}y_{n+3}+\ldots \right) \\
=\left( \frac{1}{3^{k_{n}}}-\frac{1}{3^{k_{n+1}-1}}\right) +\left( \frac{3}{%
3^{k_{n+1}}}-\frac{3}{3^{k_{n+2}-1}}\right) +\left( \frac{3^{2}}{3^{k_{n+2}}}%
-\frac{3^{2}}{3^{k_{n+3}-1}}\right) +\ldots =\frac{1}{3^{k_{n}}}.
\end{multline*}%
Hence 
\begin{multline*}
S\cdot \sum_{n=1}^{\infty }3^{n-1}\left( d_{k_{n}-1}-3d_{k_{n}}\right)
=S\cdot \sum_{n=1}^{\infty }3^{n-1}\left( \frac{r_{k_{n}-1}}{S}-3\frac{%
r_{k_{n}}}{S}\right) =\sum_{n=1}^{\infty }3^{n-1}\left(
r_{k_{n}-1}-3r_{k_{n}}\right) \\
=\sum_{n=1}^{\infty }3^{n-1}\left( x_{k_{n}}-2r_{k_{n}}\right)
=\sum_{n=1}^{\infty }3^{n-1}\left[ x_{k_{n}}-2\left(
x_{k_{n}+1}+x_{k_{n}+2}+\ldots \right) \right] \\
=\sum_{n=1}^{\infty }3^{n-1}\left[ x_{k_{n}}-2\left(
y_{n}+x_{k_{n+1}}+y_{n+1}+x_{k_{n+2}}+y_{n+2}+\ldots \right) \right] \\
=\sum_{n=1}^{\infty }3^{n-1}\left[ x_{k_{n}}-2\left(
x_{k_{n+1}}+x_{k_{n+2}}+\ldots \right) \right] -2\sum_{n=1}^{\infty
}3^{n-1}\left( y_{n}+y_{n+1}+y_{n+2}+\ldots \right) \\
=\left[ x_{k_{1}}+\left( -2+3\right) x_{k_{2}}+\left( -2-6+9\right)
x_{k_{3}}+\ldots +\right] \\
-2\left[ \left( y_{1}+3y_{2}+9y_{3}+\ldots \right) +\left(
y_{2}+3y_{3}+9y_{4}+\ldots \right) +\left( y_{3}+3y_{4}+9y_{5}+\ldots
\right) +\ldots \right] \\
=\sum_{n=1}^{\infty }x_{k_{n}}-\sum_{n=1}^{\infty }\frac{1}{3^{k_{n}}}=\frac{%
1}{2}\sum_{n=1}^{\infty }x_{k_{n}}.
\end{multline*}%
Consequently,%
\begin{eqnarray*}
\left\vert E\left( x\right) -E\left( x\right) \right\vert &=&S\cdot
\left\vert C\left( \lambda \right) -C\left( \lambda \right) \right\vert
=2S-2S\cdot \sum_{n=1}^{\infty }3^{n-1}\left( d_{k_{n}-1}-3d_{k_{n}}\right)
\\
&=&2S-\sum_{n=1}^{\infty }x_{k_{n}}=2\left( \sum_{n=1}^{\infty }\frac{1}{%
3^{n-1}}+\sum_{n=1}^{\infty }\frac{1}{3^{k_{n}}}\right) -\sum_{n=1}^{\infty }%
\frac{2}{3^{k_{n}}}=2\sum_{n=1}^{\infty }\frac{1}{3^{n-1}}=3.
\end{eqnarray*}
\end{proof}
\end{theorem}

\begin{corollary}
If a sequence $\lambda $ satisfies the assumptions of Theorem \ref{Th1} for $%
k_{0}=0$, then $\left\vert C\left( \lambda \right) -C\left( \lambda \right)
\right\vert =\frac{3}{S}$, where $S=\sum_{n=1}^{\infty }x_{n}$ and $x$ is
the sequence given by formula (\ref{3-1}).
\end{corollary}

Let $n\in \mathbb{N}$, $x_{1},x_{2},\dots ,x_{n}\in \mathbb{R}$ and $q\in
(0,1)$. The sequence 
\begin{equation*}
\left( x_{1},x_{2},\dots ,x_{n},x_{1}q,x_{2}q,\dots
,x_{n}q,x_{1}q^{2},x_{2}q^{2},\ldots ,x_{n}q^{2},\dots \right)
\end{equation*}%
is called \emph{a multigeometric sequence with the ratio }$q$ and is denoted
by $\left( x_{1},x_{2},\dots ,x_{n};q\right) $.

If a sequence $\left( k_{n}\right) $ can be decomposed into arithmetical
sequences with the same difference, then the sequence $x$ generated by $%
\left( k_{n}\right) $ is multigeometric with the ratio of the form $\frac{1}{%
3^{m}}$.

\begin{proposition}
\label{pr5}Assume that $p,m\in \mathbb{N}$, $1<s_{1}<\ldots <s_{p}\leq m$
and $k_{\left( n-1\right) p+j}=\left( n-1\right) m+s_{j}$ for $n\in \mathbb{N%
}$, $j\in \left\{ 1,\ldots ,p\right\} $, i.e. 
\begin{equation*}
k=\left( s_{1},\ldots ,s_{p},m+s_{1},\ldots ,m+s_{p},2m+s_{1},\ldots \right)
.
\end{equation*}%
Moreover, assume that $x$ and $\lambda $ are the sequences defined as in
Theorem \ref{th5}, that is 
\begin{equation*}
x_{j}=\left\{ 
\begin{array}{ccc}
\frac{2}{3^{j-1}} & \text{if} & j\in \left\{ k_{n}:n\in \mathbb{N}\right\}
\\ 
\frac{1}{3^{j-1}} & \text{if} & j\notin \left\{ k_{n}:n\in \mathbb{N}\right\}%
\end{array}%
\right. \quad \text{and}\quad \lambda _{j}=\frac{r_{j}}{r_{j-1}}.
\end{equation*}%
Then the sequence $x$ is multigeometric, 
\begin{equation*}
x=\left( \varepsilon _{1},\frac{\varepsilon _{2}}{3^{1}},\ldots ,\frac{%
\varepsilon _{m}}{3^{m-1}};\frac{1}{3^{m}}\right) \text{, where }\varepsilon
_{j}=\left\{ 
\begin{array}{ccc}
2 & \text{if} & j\in \left\{ s_{1},\ldots ,s_{p}\right\} \\ 
1 & \text{if} & j\notin \left\{ s_{1},\ldots ,s_{p}\right\}%
\end{array}%
\right. 
\end{equation*}%
and $\left\vert C\left( \lambda \right) -C\left( \lambda \right) \right\vert
=\frac{3^{m}-1}{3^{m-1}\left( x_{1}+\ldots +x_{m}\right) }$.

\begin{proof}
Since $\left\{ k_{n}:n\in \mathbb{N}\right\} =\left\{ s_{1},\ldots
,s_{p},m+s_{1},\ldots ,m+s_{p},\ldots \right\} $, the equality $x_{j}=\frac{%
\varepsilon _{j}}{3^{j-1}}$ is obvious for $j\in \left\{ 1,\ldots ,m\right\} 
$. Observe that for any $j\in \mathbb{N}$ we have $j\in \left\{ k_{n}:n\in 
\mathbb{N}\right\} $ if and only if $j+m\in \left\{ k_{n}:n\in \mathbb{N}%
\right\} $. Consequently, $\frac{x_{j+m}}{x_{j}}=\frac{3^{1-j-m}}{3^{1-j}}=%
\frac{1}{3^{m}}$ for $j\in \mathbb{N}$, which finishes the proof that $%
x=\left( \varepsilon _{1},\frac{\varepsilon _{2}}{3^{1}},\ldots ,\frac{%
\varepsilon _{m}}{3^{m-1}};\frac{1}{3^{m}}\right) $. Hence 
\begin{equation*}
S=\sum_{k=1}^{\infty }\frac{x_{1}+\ldots +x_{m}}{\left( 3^{m}\right) ^{k}}%
=3^{m}\cdot \frac{x_{1}+\ldots +x_{m}}{3^{m}-1}.
\end{equation*}%
Using Theorem \ref{tw6} we obtain $\left\vert C\left( \lambda \right)
-C\left( \lambda \right) \right\vert =\frac{3}{S}=\frac{3^{m}-1}{%
3^{m-1}\left( x_{1}+\ldots +x_{m}\right) }$.
\end{proof}
\end{proposition}

\begin{example}
\begin{enumerate}
\item Assume that $k_{n}=2n$ and the sequences $x$ and $\lambda $ are
defined as in Theorem \ref{th5}. Writing $m=2$ and $p=1$, from Proposition %
\ref{pr5}, we obtain $x=\frac{1}{3}\cdot \left( 3,2;\frac{1}{9}\right) $.
This example was considered in \cite{BGM}. Observe that $r_{2n}=%
\sum_{j=2n+1}^{\infty }x_{j}=\frac{15}{8\cdot 9^{n}}$ and $%
r_{2n+1}=\sum_{j=2n+2}^{\infty }x_{j}=\frac{7}{8\cdot 9^{n}}$. Hence $%
\lambda _{2n-1}=\frac{r_{2n-1}}{r_{2n-2}}=\frac{7}{15}$, $\lambda _{2n}=%
\frac{r_{2n}}{r_{2n-1}}=\frac{5}{21}$, and $\left\vert C\left( \lambda
\right) -C\left( \lambda \right) \right\vert =\frac{3^{2}-1}{3\left(
x_{1}+x_{2}\right) }=\frac{8}{3+2}=\frac{8}{5}$.

\item If $k_{n}=3n$, then for $m=3$ and $p=1$ we get $x=\frac{1}{9}\cdot
\left( 9,3,2;\frac{1}{27}\right) $, $r_{3n-2}=\frac{8}{13\cdot 27^{n-1}}$, $%
r_{3n-1}=\frac{99}{13\cdot 27^{n}}$, and $r_{3n}=\frac{21}{13\cdot 27^{n}}$.
Hence $\lambda _{3n-2}=\frac{r_{3n-2}}{r_{3n-3}}=\frac{8}{21}$, $\lambda
_{3n-1}=\frac{r_{3n-1}}{r_{3n-2}}=\frac{11}{24}$, $\lambda _{3n}=\frac{r_{3n}%
}{r_{3n-1}}=\frac{7}{33}$, and $\left\vert C\left( \lambda \right) -C\left(
\lambda \right) \right\vert =\frac{3^{3}-1}{9+3+2}=\frac{13}{7}$.

\item If $k=\left( 2,3,5,6,8,9,\ldots \right) $, that is $k_{2n-1}=3n-1$ and 
$k_{2n}=3n$, then for $m=3$ and $p=2$ we get $x=\frac{1}{9}\cdot \left(
9,6,2;\frac{1}{27}\right) $. Hence $\lambda _{3n-2}=\frac{25}{51}$, $\lambda
_{3n-1}=\frac{23}{75}$, $\lambda _{3n}=\frac{17}{69}$, and $\left\vert
C\left( \lambda \right) -C\left( \lambda \right) \right\vert =\frac{3^{3}-1}{%
9+6+2}=\frac{26}{17}$.
\end{enumerate}
\end{example}


\begin{thebibliography}{99}
\bibitem{AC} R. Anisca, C. Chlebovec, \emph{On the structure of arithmetic
sum of Cantor sets with constant ratios of dissection}, Nonlinearity \textbf{%
22} (2009), 2127--2140.

\bibitem{AI} R. Anisca, M. Ilie, \emph{A technique of studying sums of
central Cantor sets}, Canad. Math. Bull. \textbf{44} (2001), 12--18.

\bibitem{AS} E. Ayer, R. Strichartz, \emph{Exact Hausdorff measure and
intervals of maximal density for Cantor sets}, Trans. Amer. Math. Soc. 
\textbf{351} (1999), 3725--3741.

\bibitem{BFN} M. Balcerzak, T. Filipczak, P. Nowakowski, \emph{Families of
symmetric Cantor sets from the category and measure viewpoints}, Georgian
Math. J. \textbf{26} (2019), 545--553.

\bibitem{BBFS} T. Banakh, A. Bartoszewicz, M. Filipczak, E. Szymonik, \emph{%
Topological and measure properties of some self-similar sets}, Topol.
Methods Nonlinear Anal. \textbf{46} (2015), 1013--1028.

\bibitem{B} M. Banakiewicz, \emph{The Lebesgue measure of some M-Cantorval},
J. Math. Anal. Appl. \textbf{471} (2019), 170--179.

\bibitem{BP} M. Banakiewicz, F. Prus-Wi\'{s}niowski, \emph{M-Cantorvals of
Ferens type}, Math. Slovaca \textbf{67} (2017), 907--918.

\bibitem{BFP} A. Bartoszewicz, M. Filipczak, F. Prus-Wi\'{s}niowski, \emph{%
Topological and algebraic aspects of subsums of series}, in: \emph{%
Traditional and present-day topics in real analysis}, 345--366, \L \'{o}d%
\'{z} University Press 2013.

\bibitem{BGM} A. Bartoszewicz, S. G{\l }\c{a}b, J. Marchwicki, \emph{%
Recovering a purely atomic finite measure from its range}, J. Math. Anal.
Appl. \textbf{467} (2018), 825--841.

\bibitem{CHM} C. Cabrelli, K. Hare, U. Molter, \emph{Classifying Cantor sets
by their fractal dimensions}, Proc. Amer. Math. Soc. \textbf{138} (2010),
3965--3974.

\bibitem{DGS} D. Domanik, A. Gorodetski, B. Solomyak, \emph{Absolutely
continuous convolutions of singular measures and an applications to the
square Fibonacci Hamiltonian}, Duke Math. J. \textbf{164} (2015), 1603--1640.

\bibitem{Fl} K. J. Falconer, \emph{The Geometry of Fractal Sets}, Cambridge
University Press, Cambridge, 1985.

\bibitem{Fe} C. Ferens, \emph{On the range of purely atomic measures},
Studia Math. \textbf{77 }(1984), 261-263.

\bibitem{GN} J. A. Guthrie, J. E. Nymann, \emph{The topological structure of
the set of subsums of an infinite series}, Colloq. Math. \textbf{55} (1988),
323-327.

\bibitem{GZ} I. Garcia, L. Zuberman, \emph{Exact packing measure of central
Cantor sets in the line}, J. Math. Anal. Appl. \textbf{386} (2012), 801--812.

\bibitem{H} M. Hall, \emph{On the sum and product of continued fractions},
Ann. Math. \textbf{48} (1947), 966--993.

\bibitem{HZ} K. E. Hare, L. Zuberman, \emph{Classifying Cantor sets by their
multifractal spectrum}, Nonlinearity \textbf{23} (2010), 2919--2933.

\bibitem{HKY} B. Hunt, I. Kan, J. Yorke, \emph{Intercection of thick Cantor
sets}, Trans. Amer. Math. Soc. \textbf{339} (1993), 869--888.

\bibitem{Ka} S. Kakeya, \emph{On the partial sums of an infinite series}, T%
\^{o}hoku Sci. Rep. \textbf{3} (1914), 159--164.

\bibitem{K1} R. L. Kraft, \emph{Random intersections of thick Cantor sets},
Trans. Amer. Math. Soc. \textbf{352} (2000), 1315--1328.

\bibitem{K} R. L. Kraft, \emph{What's the difference between Cantor sets?},
Amer. Math. Monthly \textbf{101} (1994), 640--650.

\bibitem{MO} P. Mendes, F. Oliveira, \emph{On the topological structure of
the arithmetic sum of two Cantor sets}, Nonlinearity \textbf{7} (1994),
329--343.

\bibitem{N} Z. Nitecki, \emph{Cantorvals and Subsum Sets of Null Sequences},
Amer. Math. Monthly \textbf{122} (2015), 862--870.

\bibitem{NS} J. E. Nymann, R. A. S\'{a}enz, \emph{On the paper of Guthrie
and Nymann on subsums of infinite series}, Colloq. Math. \textbf{83} (2000),
1-4.

\bibitem{PT} J. Palis, F. Takens, \emph{Hyperbolicity and Sensitive Chaotic
Dynamics of Homoclinic Bifurcations}, Cambridge University Press, Cambridge,
1993.

\bibitem{PP} S. Pedersen, J. D. Philips, \emph{Exact Hausdorff measure of
certain non-self-similar Cantor sets}, Fractals \textbf{21} (2013), 1350016.

\bibitem{FF} F. Prus-Wi\'{s}niowski, F. Tulone, \emph{The arithmetic
decomposition of central Cantor sets}, J. Math. Anal. Appl. \textbf{467}
(2018), 26--31.

\bibitem{S} A. Sannami,\emph{\ An example of a regular Cantor set whose
difference set is a Cantor set with positive measure}, Hokkaido Math. J. 
\textbf{21} (1992), 7--24.

\bibitem{Ta} Y. Takahashi, \emph{Products of two Cantor sets}, Nonlinearity 
\textbf{30} (2017), 2114--2137.

\bibitem{T16} Y. Takahashi, \emph{Quantum and spectral properties of the
Labyrinth model}, J. Math. Phys. \textbf{57} (2016), 063506.

\bibitem{T19} Y. Takahashi, \emph{Sums of two self-similar Cantor sets}, J.
Math. Anal. Appl. \textbf{477} (2019), 613--626.

\bibitem{WS} A. D. Weinstein, B. E. Shapiro, \emph{On the structure of the
set of }$\overline{\alpha }$\emph{-representable numbers}, Izv. Vyssh.
Uchebn. Zaved. Mat. \textbf{24} (1980), 8-11.
\end{thebibliography}
\end{document}